\documentclass[a4paper,reqno,11pt]{amsart}
\usepackage{amsfonts,amsmath,amsthm,amssymb,stmaryrd}
\usepackage{mathrsfs}

\usepackage[left=1 in, right=1 in,top=1 in, bottom=1 in]{geometry}

\allowdisplaybreaks

\numberwithin{equation}{section}
\usepackage{color}

\newtheorem{theorem}{Theorem}[section]
\newtheorem{lemma}[theorem]{Lemma}
\newtheorem{proposition}[theorem]{Proposition}

\newtheorem{remark}[theorem]{Remark}

\theoremstyle{definition}

\newcommand{\Div}{\operatorname{div}}
\newcommand{\curl}{\operatorname{curl}}

\renewcommand{\epsilon}{\varepsilon}

\providecommand{\abs}[1]{\left\vert#1\right\vert}
\providecommand{\norm}[1]{\left\Vert#1\right\Vert}

\providecommand{\norm}[1]{\left\Vert#1\right\Vert}

\def\dt{\partial_t}

\def\pa{\partial}

\def\RRvert2{\right \vert\! \right\vert}
\def\Lvert3{\left \vert\!\left\vert\!\left\vert}
\def\Rvert3{\right \vert\!\right\vert\!\right\vert}

\def\dt{\partial_t}

\def\hal{\frac{1}{2}}

\def\ls{\lesssim}

\def\a{\mathcal{A}}

\def\j{\mathcal{J}}

\def\fj1{\mathcal{J}^{-1}}

\def\bp{\partial_z}

\def\ak{{\bar{\mathcal{A}}}}

\makeindex

\title[FBP of MHD]{Well-posedness of axially symmetric incompressible ideal magnetohydrodynamic equations with vacuum under the non-collinearity condition}

\author[X. Gu]{Xumin Gu}
\address[X. Gu]{Department of Mathematics, Shanghai University of Finance and Economics;
Shanghai Center of Mathematical Sciences, People's Republic of China}
\email{gu.xumin@shufe.edu.cn}

\begin{document}


\begin{abstract}
We consider a free boundary problem for the axially symmetric incompressible ideal magnetohydrodynamic equations that describes the motion of the plasma in vacuum. Both the plasma magnetic field and vacuum magnetic field are tangent along the plasma-vacuum interface. Moreover, the vacuum magnetic field is composed in a non-simply connected domain and hence is non-trivial. Under the non-collinearity condition on the free surface, we prove the local well-posedness of the problem in Sobolev spaces.
\end{abstract}

\maketitle

\section{Introduction}
\subsection{Eulerian formulation}
In this paper, we consider the free boundary problem of the axially symmetric incompressible ideal MHD equations:
\begin{equation}
  \label{mhd}
\begin{cases}
\partial_tu^r + (u^r \partial_r+u^z\partial_z) u^r-\dfrac{(u^{\theta})^2}{r} + \partial_r (P+\dfrac{1}{2}\abs{B}^2) = (B^r\partial_r+B^z\partial_z)B^r -\dfrac{(B^{\theta})^2}{r}&\text{in}\ \  \Omega(t),\\
\partial_tu^{\theta} + (u^r \partial_r+u^z\partial_z) u^{\theta}+\dfrac{u^{\theta}u^r}{r} = (B^r\partial_r+B^z\partial_z)B^{\theta} +\dfrac{B^{\theta}B^r}{r}&\text{in}\ \  \Omega(t),\\
\partial_tu^z + (u^r \partial_r+u^z\partial_z) u^z+ \partial_z (P+\dfrac{1}{2}\abs{B}^2) = (B^r\partial_r+B^z\partial_z)B^z &\text{in}\ \ \Omega(t),\\
\partial_tB^r +  (u^r \partial_r+u^z\partial_z) B^r = (B^r\partial_r+B^z\partial_z) u^r &\text{in}\ \  \Omega(t),\\
\partial_tB^{\theta} +  (u^r \partial_r+u^z\partial_z) B^{\theta}+\dfrac{B^ru^{\theta}}{r}= (B^r\partial_r+B^z\partial_z) u^{\theta}+\dfrac{u^rB^{\theta}}{r} &\text{in}\ \ \Omega(t),\\
\partial_tB^z +  (u^r \partial_r+u^z\partial_z) B^z = (B^r\partial_r+B^z\partial_z) u^z &\text{in}\ \ \Omega(t),\\
\partial_r u^r+\dfrac{u^r}{r}+\partial_z u^z =0 &\text{in}\ \ \Omega(t),\\
\partial_r B^r+\dfrac{B^r}{r}+\partial_z B^z =0 &\text{in}\ \ \Omega(t).
\end{cases}
\end{equation}
In the equation \eqref{mhd}, $u(t,x)=u^r(t,r,z)e_r+u^{\theta}(t,r,z)e_{\theta}+u^z(t,r,z)e_z$ is the Eulerian or spatial velocity field, $B=B^r(t,r,z)e_r+B^{\theta}(t,r,z)e_{\theta}+B^z(t,r,z)e_z$ is the magnet field, and $P$ denotes the pressure function of the fluid which occupies the moving vessel domain:
$$\Omega(t): \{(x_1, x_2, z)| 0\leq r< r(z, t), z\in \mathbb T\}.$$
Here $e_r=(\cos\theta,-\sin\theta,0), e_{\theta}=(-\sin\theta,\cos\theta,0), e_z=(0,0,1), r=\sqrt{x_1^2+x_2^2}, z=x_3$, $\theta=\arctan{\frac{x_2}{x_1}}$.
We require the following boundary condition on the free surface $\Gamma(t):=\partial\Omega(t)$:
\begin{equation}
\label{kbd}
V(\Gamma(t))=u\cdot n \,\,\text{on}\,\, \Gamma(t)
\end{equation}
and
\begin{equation}
\label{hbd}
B\cdot n=0,\,\,P+\dfrac{1}{2}\abs{B}^2=\dfrac{1}{2}\dfrac{C^2(t)}{r^2}, \,\,\text{on}\,\,\Gamma(t).
\end{equation}
The equation \eqref{kbd} is called the kinematic boundary condition which states that the free surface $\Gamma(t)$ moves with the velocity of the fluid, where $V(\Gamma(t))$ denote the normal velocity of $\Gamma(t)$ and $n$ is the outward normal of the domain $\Omega(t)$. The first condition of the equation \eqref{hbd} means the fluid is perfect conductor. The second condition expresses the continuity of pressure on the free interface and $C(t)$ is given by
\begin{equation}
\label{cformula}
C(t)=C(0)e^{\int_0^tA(\tau)\,d\tau}, \,\,A(t)=\dfrac{\int_{\mathbb T}(u\cdot n)\sqrt{1+(\partial_z r(z,t))^2}\,dz}{\int_{\mathbb T}(\ln R_S-\ln r(z,t))\,dz},
\end{equation}
$R_S$ is a constant larger than $r(z,t)$.

The system \eqref{mhd}-\eqref{hbd} can be used to describe the motion of the plasma confined inside a rigid wall and isolated from it by vacuum, which is one of laboratory plasma model problems (see \cite[Chapter 4.6.1]{Go_04}). In the general setting, the plasma region $\Omega(t)$ is surrounded by the vacuum region $\Omega_v(t)$, and the moving plasma-vacuum interface $\Gamma(t)=\partial\Omega(t)$ does not intersect with the outer wall $\partial\Omega_w$, where $\Omega_w=\Omega(t)\cup \Gamma(t)\cup \Omega_v(t)$ is a fixed domain. When the characteristic plasma velocity is very small compared to the speed of sound, the motion of the plasma is governed by the incompressible ideal MHD in $\Omega(t)$, i.e., \eqref{mhd}. In the vacuum region $\Omega_v(t)$, we neglect the displacement current in the Maxwell equations as usual in the non-relativistic MHD and assume the pre-Maxwell equations:
\begin{equation}
\label{premax}
\begin{cases}
\curl \mathcal{B} =0,\quad \Div \mathcal{B} =0 &\text{in }\Omega_v(t),\\
\curl \mathcal{E} =-\dt \mathcal{B},\quad \Div \mathcal{E} =0&\text{in }\Omega_v(t).
\end{cases}
\end{equation}
In the equations \eqref{premax}, $\mathcal{B}$ and $\mathcal{E}$ denotes the magnetic and electric fields in vacuum, respectively. The motion of the plasma is connected with the vacuum through the jump condition on the interface $\Gamma(t)$:
\begin{equation*}
  \left(\big(P  + \dfrac{1}{2}|B|^2\big)I  - B\otimes B\right)n=\left(  \dfrac{1}{2}|\mathcal{B}|^2I  -\mathcal{B}\otimes\mathcal{ B}\right)n\quad\text{on } \Gamma(t)
\end{equation*}
and
\begin{equation}
\label{ebd}
 (B-\mathcal{B})\cdot n  = 0,\quad   (E-  \mathcal{E}) \times n=  -(u \cdot n) (B- \mathcal{B}) \quad\text{on } \Gamma(t).
\end{equation}
Here  $E$ is the electric field of the plasma, i.e.,
\begin{equation}
\label{efor}
E=-u\times B.
\end{equation}
 We also require the jump condition on $\partial\Omega_w$:
\begin{equation*}
 (\mathcal{B}-\hat{\mathcal{B}})\cdot \nu  = 0,\quad (\mathcal E-\hat{\mathcal E})\times \nu=0 \quad\text{on } \partial\Omega_w.
\end{equation*}
Here $\hat{\mathcal{B}}$ and $\hat{\mathcal E}$ denotes the magnetic and electric fields outside the wall $\partial\Omega_w$ and $\nu$ is the outward unit normal of $\partial\Omega_w$.

The well-posedness of the general plasma-vacuum interface problem is still an open question. In this paper we restrict to an axially symmetric case and $\Omega_w$ is a cylinder domain: $\{(x_1, x_2, z)|0\leq r< R_S, z\in \mathbb T\}$, $\pa\Omega_w:=\{(x_1,x_2, z)|r=R_S, z\in \mathbb T\}$ is a perfectly conducting wall and when the plasma is a perfect conductor. In this setting, we have the following boundary conditions
\begin{equation}
\label{B1}
\mathcal{B}\cdot \nu=0,\quad \mathcal{E}\times\nu=0, \quad\text{on } \partial\Omega_w
\end{equation}
and
\begin{equation}
\label{B2}
B\cdot n=\mathcal{B}\cdot n=0\quad \text{on } \Gamma(t).
\end{equation}

Now with the first equation of \eqref{B1} and the first equation of \eqref{B2}, we can derive a formula for the vacuum magnet filed $\mathcal B=\mathcal B^{r}(r,z,t)e_{r}+\mathcal B^{\theta}(r,z,t) e_{\theta}+\mathcal B^z(r,z,t)e_z$. In fact, we can transfer the first equation of \eqref{premax} and the first equation of \eqref{B1} into the two decoupled system:
\begin{equation*}
	\partial_r\mathcal B^{\theta}+\dfrac{\mathcal B^{\theta}}{r}=0,\,\,
\partial_z\mathcal B^{\theta}=0,
\end{equation*}
with no boundary condition and
\begin{equation*}
\begin{cases}
\partial_r\mathcal B^r+\dfrac{1}{r}\mathcal B^r+\partial_z \mathcal B^z=0,\\
\partial_r\mathcal B^r-\partial_z\mathcal B^z=0,
\end{cases}
\end{equation*}
with boundary condition
\begin{equation*}
\begin{split}
\mathcal B^r+\mathcal B^z\partial_z r=0, \,\,&\text{on}\,\, \Gamma(t): r=r(z,t),\\
\mathcal B^r=0, \,\,&\text{on}\,\, r=R_S.
\end{split}
\end{equation*}
Then we see
\begin{equation}
\mathcal B^r=\mathcal B^z=0, \mathcal B^{\theta}=\dfrac{C(t)}{r}
\label{vab}
\end{equation} are solutions to the above two systems.
In order to determine $C(t)$, we consider the elliptic system of the vacuum electronic field:
\begin{equation*}
\begin{cases}
\nabla\times \mathcal E=-\dfrac{C'(t)}{r},\,\,&\text{in}\,\,\Omega_v(t),\\
\nabla\cdot \mathcal E=0, \,\,&\text{in}\,\,\Omega_v(t),\\
n\times \mathcal E= (u\cdot n)\dfrac{C(t)}{r} \,\,&\text{on}\,\, \Gamma(t), \\\nu\times \mathcal E=0 \,\,&\text{on}\,\, \pa \Omega_w.
\end{cases}
\end{equation*}
The third equation is obtained by using the first equation of \eqref{ebd} and the equation \eqref{efor}.
Now by using integration by parts, we have
\begin{equation*}
\int_{\Omega^-}\nabla \times \mathcal E\cdot \mathcal B\,dV=-\int_{\Gamma}n\times \mathcal E\cdot \mathcal B\,d\sigma,
\end{equation*}
where we used $\nabla \times \mathcal B=0$ in $\Omega_v(t)$ and $\nu\times \mathcal E=0$ on $\pa \Omega_w$. Thus, we have
\begin{equation*}
C'(t)\int_{\Omega^-}\dfrac{1}{r^2}\,dV=C(t)\int_{\Gamma}\dfrac{(u\cdot n)}{r}\,d\sigma
\end{equation*}
which gives
\begin{equation*}
C(t)=C(0)e^{\int_0^tA(\tau)\,d\tau}, \,\,A(t)=\dfrac{\int_{\Gamma}\dfrac{(u\cdot n)}{r}\,d\sigma}{\int_{\Omega^-}\dfrac{1}{r^2}\,dV}=\dfrac{\int_{\mathbb T}(u\cdot n)\sqrt{1+(\partial_z r(z,t))^2}\,dz}{\int_{\mathbb T}(\ln R_S-\ln r(z,t))\,dz}.
\end{equation*}

Hence the plasma-vacuum interface problem reduces to the free boundary problem \eqref{mhd}--\eqref{hbd}. Our purpose of this paper is to establishing the local well-posedness for this problem.
\begin{remark}
If we set $C(0)=0$ in the formula \eqref{cformula}, then we have $\mathcal B=0$ for all time. With this trivial vacuum magnet field, the local well-posedness was proved in \cite{GW_16} without axially symmetric assumption. In fact, if the vacuum domain $\Omega_v$ is simply connected, we only have trivial magnet field. It is interesting to study the non-trivial vacuum magnet field and its interaction with plasma magnet field, this is also the main reason for us to consider the problem with a non-simply-connected vacuum domain in this paper.
\end{remark}
\subsection{Lagrangian reformulation}
We tranform the Eulerian problem \eqref{mhd}-\eqref{hbd} on the moving domain $\Omega(t)$ to be one on the fixed domain $\Omega$ by the use of Lagrangian coordinates.
Let $x\in\Omega$ be the Lagrangian coordinate and $\eta(x,t)$ be the Eulerian coordinate, which means $\eta(x,t)\in\Omega(t)$ denote the "position" of the fluid particle $x$ at $t$. Thus,
 $$\partial_t\eta(x,t)=u(\eta(x,t),t)\,\,\text{for}\,\, t>0,\,\,\eta(x,0)=x.$$
Now we denote $R(x,t)=\sqrt{(\eta_1)^2+(\eta_2)^2}, \Theta(x,t)=\arctan \dfrac{\eta_2}{\eta_1}, Z(x,t)=\eta_3$, and $r=\sqrt{(x_1)^2+(x_2)^2}$, $\theta=\arctan\dfrac{x_2}{x_1}$, $z=x_3$. Then we can derive
\begin{equation}
	\label{rst}
	\begin{cases}
&\partial_tR=u^r\left(R(x,t),Z(x,t),t\right),\\
&\partial_t\Theta=\dfrac{1}{R(x,t)}u^{\theta}\left(R(x,t),Z(x,t),t\right),\\
&\partial_tZ=u^z\left(R(x,t),Z(x,t),t\right).
\end{cases}
\end{equation}
and
\begin{equation}
R(x,0)=r, \Theta(x,0)=\theta, Z(x,0)=z.
	\label{inz}
\end{equation}
From the first and third equation of \eqref{rst} and the initial data \eqref{inz}, we have
\begin{equation}
	\label{rzRz}
	R(x,t)=R(r,z,t), \quad Z(x,t)=Z(r,z,t).
\end{equation}
And for $\Theta$, we have \begin{equation}
	\label{dtheta}
	\Theta(r,\theta, z, t)=\theta+\int_0^t\dfrac{u^{\theta}(R,Z,\tau)}{R}\,d\tau=\theta+\hat\Theta(r,z,t).
\end{equation}
Now we denote
\begin{equation}
	\label{deflag}
	\begin{cases}
	&v^r(r,z,t)=u^r\left(R, Z, t\right), v^{\theta}(r,z,t)=u^{\theta}\left(R, Z, t\right), v^z(r,z,t)=u^z\left(R, Z, t\right),\\
	&b^r(r,z,t)=B^r\left(R, Z, t\right), b^{\theta}(r,z,t)=B^{\theta}\left(R, Z, t\right), b^z(r,z,t)=B^z\left(R, Z, t\right),\\&q(r,z,t)=P\left(R, Z, t\right)+\dfrac{1}{2}|B|^2\left(R, Z, t\right),
\end{cases}
\end{equation}
and denote the deformation tensor between $(R, Z)$ and $(r, z)$ as $\mathcal{F}_{ij}=\partial_{a_j}\zeta^i(r,z,t)$, where $\zeta=(R, Z), a=(r, z)$, (e.g. $\mathcal F_{11}=\partial_r R$).
Then we have the following Lagrangian version of \eqref{mhd} in the fixed reference domain $\Omega$:
\begin{equation}
	\label{eq:mhdo}
\begin{cases}
	\partial_t v^r-\dfrac{(v^{\theta})^2}{R}+\pa^\a_{R} q
=(b^r\pa^\a_R+b^z\pa^\a_Z)b^r-\dfrac{(b^{\theta})^2}{R}\,\,&\text{in}\,\,\Omega,\\
\partial_t v^{\theta}+\dfrac{v^{\theta}v^r}{R}
=(b^r\pa^\a_R+b^z\pa^\a_Z)b^{\theta}+\dfrac{b^rb^{\theta}}{R}\,\,&\text{in}\,\,\Omega,\\
\partial_t v^z+\pa^\a_Z q
=(b^r\pa^\a_R+b^z\pa^\a_Z)b^z\,\,&\text{in}\,\,\Omega,\\
\partial_tb^r =(b^r\pa^\a_R+b^z\pa^\a_Z) v^r\,\,&\text{in}\,\,\Omega,\\
\partial_tb^{\theta} +\dfrac{b^rv^{\theta}}{R}=(b^r\pa^\a_R+b^z\pa^\a_Z) v^{\theta}+\dfrac{v^rb^{\theta}}{R}\,\,&\text{in}\,\,\Omega,\\
\partial_tb^z =(b^r\pa^\a_R+b^z\pa^\a_Z) v^z\,\,&\text{in}\,\,\Omega,\\
\pa^\a_R(Rv^r)+\pa^\a_Z(Rv^z)=0\,\,&\text{in}\,\,\Omega,\\
\pa^\a_R(Rb^r)+\pa^\a_Z(Rb^z)=0\,\,&\text{in}\,\,\Omega,\\
 (f ,v, b)|_{t=0}=(\text{Id}, v_0, b_0).
\end{cases}
\end{equation}
where $\a=\mathcal{F}^{-T}$, $\pa^\a_{\zeta^i}:=\a_{ij}\partial_{a_j}$.
Here we donote
	\begin{equation}
		\Omega:=\{(x_1,x_2,x_3)|(x_1)^2+(x_2)^2< R_0, x_3\in \mathbb T\}
		\label{domain}
	\end{equation}
Two dynamic boundary conditions become:
\begin{equation*}
\begin{cases}
	q=\dfrac{1}{2}\dfrac{C^2(t)}{R^2}\,\,&\text{on}\,\,\Gamma\times(0,T],\\
(b^r,b^z) \a N =0\,\,&\text{on}\,\,\Gamma\times(0,T]\label{hebd}
\end{cases}
\end{equation*}
where $\Gamma:=\{(x_1,x_2, z)|(x_1)^2+(x_2)^2= R_0, z\in \mathbb T\}$, $N=(1,0)$ and
$$C(t)=C(0)e^{\int_0^t A(\tau)\,d\tau},
	A(t)=\dfrac{\int_{\mathbb T}\left(v^r\partial_z Z(R_0,z,t)-v^z\partial_z R(R_0,z,t)\right)\,dz}{\int_{\mathbb T}\left(\ln R_S-\ln R(R_0,z,t)\right)\partial_z Z\,dz}.$$

Now we follow the idea used in \cite{GW_16} to transfer the system \eqref{eq:mhdo} to a free-surface incompressible Euler system with a forcing term induced by the flow map. In fact, it can be checked directly that
\begin{equation*}
	\partial_t\left((b^r,b^z)\a\right)=0.
\end{equation*}
Then we have
\begin{equation}
	\label{hformula}
	\begin{cases}
	b^r=(b_0^r\partial_r+b_0^z\partial_z) R,\\
	b^z=(b_0^r\partial_r+b_0^z\partial_z) Z.
 \end{cases}
\end{equation}
With this equality, it can be checked that if $\partial_r(rb_0^r)+\partial_z(rb_0^z)=0$ in $\Omega$ and $b_0^r=0$ on $\Gamma$, then the condition $\pa^\a_R(Rb^r)+\pa^\a_Z(Rb^z)=0$ and the second condition of \eqref{hebd} are satiesfied naturally. Then we plug \eqref{hformula} into the equation for $b^{\theta}$, we have
\begin{equation*}
	\partial_t b^{\theta}-\dfrac{v^rb^{\theta}}{R}=(b_0^r\partial_r+b_0^z\partial_z)v^{\theta}-\dfrac{v^{\theta}(b_0^r\partial_r+b_0^z\partial_z)R}{R}
\end{equation*}
Using \eqref{rst}, we have
\begin{equation*}
\partial_t\left(\dfrac{b^{\theta}}{R}\right)=\partial_t\left(\left(b_0^r\partial_r+b_0^z\partial_z \right)\Theta\right)
\end{equation*}
and hence
\begin{equation}
	\label{hfor2}
b^{\theta}=R(b_0^r\partial_r+b_0^z\partial_z)\Theta+\dfrac{Rb_0^{\theta}}{r}.
\end{equation}
Thus, with \eqref{hformula} and \eqref{hfor2}, we arrive at:
\begin{equation}
	\label{eq:mhd}
\begin{cases}
	\partial_t R= v^r\,\,&\text{in}\,\,\Omega,\\
		\partial_t Z = v^z\,\,&\text{in}\,\,\Omega,\\
\partial_t v^r-\dfrac{(v^{\theta})^2}{R}+\pa^\a_R q
=(b_0\cdot\nabla)^2R-R(b_0\cdot\nabla \Theta)^2\,\,&\text{in}\,\,\Omega,\\
\partial_t v^z+\pa^\a_Z q
=(b_0\cdot\nabla)^2Z\,\,&\text{in}\,\,\Omega,\\
\pa^\a_Rv^r+\dfrac{v^r}{R}+\pa^\a_Zv^z=0\,\,&\text{in}\,\,\Omega,\\
\partial_t \Theta = \dfrac{v^{\theta}}{R}\,\,&\text{in}\,\,\Omega,\\
\partial_t v^{\theta}+\dfrac{v^{\theta}v^r}{R}
=(b_0\cdot\nabla)(Rb_0\cdot\nabla\Theta)+b_0\cdot\nabla Rb_0\cdot\nabla\Theta\,\,&\text{in}\,\,\Omega,\\
(R,\Theta, Z, v)|_{t=0}=(r, \theta, z, u_0).
\end{cases}
\end{equation}
with boundary condition:
\begin{equation}
	q=\dfrac{C^2(t)}{R^2}\quad \text{on } \Gamma, C(t)=C(0)e^{\int_0^t }A(\tau)\,d\tau
	\label{qbdd}
\end{equation}
where $b_0\cdot\nabla=b_0^r\partial_r+b_0^z\partial_z+\dfrac{1}{r}b_0^{\theta}\partial_{\theta}$, $A(t)=\dfrac{\int_{\mathbb T}\left(v^r\partial_z Z-v^z\partial_z R(R_0,z,t)\right)\,dz}{\int_{\mathbb T}\left(\ln R_S-\ln R(R_0,z,t)\right)\partial_zZ\,dz}.$
In the system \eqref{eq:mhd}, the initial magnet field $b_0$ can be regarded as a parameter vector that satisfies
\begin{equation}
\label{bcond}
\partial_rb_0^r+\dfrac{1}{r}b_0^r+\partial_zb_0^z=0\,\,\text{in}\,\, \Omega \,\,\text{and}\,\, b_0^r=0  \,\,\text{on}\,\, \Gamma.
\end{equation}
\subsection{Previous works}
 Free boundary problems in fluid mechanics have important physical background and have been studied intensively in the mathematical community. There are a huge amount of mathematical works, and we only mention briefly some of them below that are closely related to the present work,  that is, those of the incompressible Euler equations and the related ideal MHD models.

For the incompressible Euler equations, the early works were focused on the irrotational fluids, which began
with the pioneering work of Nalimov \cite{N} of the local well-posedness for the small initial data and was generalized to the
general initial data by the breakthrough of Wu \cite{Wu1,Wu2} (see also Lannes \cite{Lannes}). For the irrotational inviscid fluids, certain dispersive effects can be used to establish the global well-posedness for the small initial data; we refer to  Wu \cite{Wu3,Wu4}, Germain, Masmoudi and Shatah \cite{GMS1, GMS2}, Ionescu and Pusateri \cite{IP,IP2} and Alazard
and Delort \cite{AD}. For the general incompressible Euler equations, the first local well-posedness in 3D was obtained by Lindblad \cite{Lindblad05} for the case without surface tension (see Christodoulou and Lindblad \cite{CL_00} for the a priori estimates) and
by Coutand and Shkoller \cite{CS07} for the case with (and without) surface tension. We also refer to the results of Shatah and Zeng \cite{SZ} and Zhang and Zhang
\cite{ZZ}. Recently, the well-posedness in conormal Sobolev spaces can be found by the the inviscid limit of the free-surface incompressible Navier-Stokes equations, see Masmoudi and Rousset \cite{MasRou} and  Wang and Xin \cite{Wang_15}.

The study of free boundary problems for the ideal MHD models seems far from being complete; it attracts many research interests, but up to now only few well-posedness theory for the nonlinear problem could be found. For the plasma-vacuum interface model that a surface current $J$ is added as an outer force term to the vacuum pre-Maxwell system \eqref{premax}, with the non-collinearity condition holding for two magnet fields on the boundary:
\begin{equation}
\label{st}
\abs{B\times \mathcal B}>0\,\,\text{on}\,\,\Gamma(t),
\end{equation} the well-posedness
of the nonlinear compressible problem was proved in Secchi and Trakhinin \cite{Secchi_14} by
the Nash-Moser iteration based on the previous results on the linearized problem \cite{Trak_10,Secchi_13}. The well-posedness of the linearized incompressible problem was proved by Morando, Trakhinin and Trebeschi \cite{Mo_14}, the nonlinear incompressible problem was solved by Sun, Wang and Zhang \cite{Sun_17} very recently.  On the other hand, Hao and Luo \cite{Hao_13} established
a priori estimates for the incompressible plasma-vacuum interface problem under the Taylor’s sign condition:
\begin{equation}
\label{taylor}
\dfrac{\partial\left(P+\dfrac{1}{2}\abs{B}^2-\dfrac{1}{2}\abs{\mathcal B}^2\right)}{\partial{n}}\leq -\epsilon<0\,\,\text{on}\,\,\Gamma(t),
\end{equation}
under the assumption that the strength of the magnetic field is constant on the
free surface by adopting a geometrical point of view \cite{CL_00}. Recently, Gu and Wang proved the well-posedness of the incompressible plasma-vacuum problem under \eqref{taylor} with the vacuum magnet field is zero and the well-posedness of the axially symmertic ideal MHD equation \eqref{mhd} under \eqref{taylor} will be addressed in the forth coming paper. However, without axially symmetric assumption, the well-posedness of the plasma-vacuum interface problem under \eqref{taylor} is still unknown when the vacuum magnetic field $\mathcal B$ is non trivial. Finally, we also mention some works about the current-vortex sheet problem, which describes a velocity and magnet field discontinuity in two ideal MHD flows. The nonlinear stability of compressible current-vortex sheets was solved independently by Chen and Wang \cite{Chen_08} and Trakinin \cite{Trak_09} by using the Nash-Moser iteration. For incompressible current-vortex sheets, Coulombel, Morando, Secchi and Trebeschi \cite{CMST} proved an a priori estimate for the nonlinear problem under a strong stability condition, and  Sun, Wang and Zhang \cite{Sun_15} solved the nonlinear stability.

\section{Main Result}
Before stating our results of this paper, we may refer the readers to our notations and conveniences in Section \ref{Notation}.

We define the higher order energy functional
\begin{equation}
\label{edef}
\mathfrak{E}(t)=\norm{(v^r, v^{\theta}, v^z)}_4^2+\norm{(R, Z)}_4^2+\norm{(b_0\cdot\nabla R, Rb_0\cdot\nabla \Theta, b_0\cdot\nabla Z)}_4^2
\end{equation}
Then the main result in this paper is stated as follows.
\begin{theorem}\label{mainthm}
Suppose that the initial data $ (v_0^r, v_0^{\theta}, v_0^z) \in H^4_{r,z}(\Omega)$ with $v^r_0(0,z)=v^{\theta}_0(0,z)=0$ and $\partial_rv^r_0+\frac{1}{r}v^r_0+\partial_zv^z_0=0$ and $(b_0^r, b_0^{\theta}, b_0^z) \in H^4_{r,z}(\Omega)$ satisfies \eqref{bcond} and that
\begin{equation}
\label{stc}
\abs{b_0^z}\geq\delta>0\,\,\text{on}\,\,\Gamma
\end{equation}
 holds initially. Then there exists a $T_0>0$ and a unique solution $(v^r, v^{\theta}, v^z, q, R, \Theta, Z)$ to  \eqref{eq:mhd} on the time interval $[0, T_0]$ which satisfies
\begin{equation}\label{enesti}
\sup_{t \in [0,T_0]} \mathfrak E(t) \leq P\left(\norm{\left(v_0^r, v_0^{\theta}, v_0^z\right)}_4^2+\norm{\left(b_0^r, b_0^{\theta}, b_0^z\right)}_4^2\right),
\end{equation}
where $P$ is a generic polynomial.

\end{theorem}
\begin{remark}
Recall the vacuum magnet field formula \eqref{vab},  the inequality \eqref{stc} is actually the non-collinearity condition \eqref{st} under our axially symmetric settings.
\end{remark}
\subsection{Strategy of the proof}
The strategy of proving the local well-posedness for the inviscid free boundary problems consists of three main parts: the a priori estimates in certain energy functional spaces, a suitable approximate problem which is asymptotically consistent with the a priori estimates, and the construction of solutions to the approximate problem.  For the incompressible MHD equations \eqref{eq:mhd}, we derive our a priori estimates in the following way. First, we divide \eqref{eq:mhd} into two sub-systems: one is for $(v^r, v^z, q, R, Z)$, the other one is for $(v^{\theta}, \Theta)$ (see \eqref{sub1} and \eqref{sub2}). The a priori estimates for $(v^{\theta}, \Theta)$ can be obtained by standard energy method. This is because there is no pressure in this subsystem and no boundary integral needs to be considered. Here, one will meet the difficulty to deal with the singularity brought by the cylinder coordinates, i.e. the estimates of $\dfrac{v^{\theta}}{r}$. Hence, we will apply the high order Hardy inequality to control these terms. On the other hand, the estimates for $(v^r, v^z, q, R, Z)$ is more complicated. We shall use tangential energy estimates combining with divergence and curl estimates to close the a priori estimates of $(\nu, q, \zeta)$, where we denote $\nu=(v^r, v^z), \zeta=(R, Z)$. During this process, there are several difficulties to deal with. In the usual derivation of the a priori tangential energy estimates of \eqref{sub1} in the $H^4_{r,z}$ setting, one deduces
\begin{equation}
\begin{split}
&\hal\dfrac{d}{dt} \int_{\Omega}\abs{\bp^4 \nu}^2+\abs{\bp^4 ( b_0\cdot\nabla\zeta)}^2+\underbrace{\int_{\Gamma}-\a_{mj}\partial_{a_j} q\bp^4{\zeta}_m\bp^4 \nu_i{\mathcal A}_{i1}+\bp^4q\bp^4\nu_i\a_{i1}}_{\mathcal{I}_b}
\\&\quad \approx \underbrace{\int_\Omega \bp^4 D\zeta \bp^4  \zeta+ \bp^4 \nabla\zeta\bp^4  q}_{\mathcal{R}_Q}+l.o.t.,
\end{split}
\end{equation}
The first difficulty one will meet is the loss of derivatives in estimating $\mathcal{R}_Q$ (by recalling the energy functional $\mathfrak{E}(t)$ defined by \eqref{edef}).  Our idea to overcome this difficulty is, motivated by \cite{MasRou,Wang_15, GW_16}, to use Alinhac's good unknowns $
\mathcal{V} =\bp^4 \nu- \bp^4\zeta \cdot\nabla_\a \nu$ and $ \mathcal{Q} =\bp^4 q- \bp^4\zeta \cdot\nabla_\a q$, which derives a crucial cancellation observed by Alinhac \cite{Alinhac}, i.e., when considering the equations for $\mathcal{V} $ and $Q$, the term $\mathcal{R}_Q$ disappears. The second difficulty is to estimate the boundary integral $\mathcal{I}_b$. Recalling the boundary condition for $q$, one needs to control $\abs{\bp^4\zeta}_{1/2}$, which means a loss of derivatives again. To overcome this difficulty, we use the following important observation: with the non-collinearity condition \eqref{stc} and boundary condition \eqref{bcond}, one can have
$$\abs{\bp^4\zeta}_{1/2}\ls \abs{\dfrac{\bp^3(b_0^z\bp\zeta)}{b_0^z}}_{1/2}\ls \norm{b_0\cdot\nabla\zeta}_4.$$
This means the non-collinearity condition \eqref{stc} can actually improve one order boundary regularity, which plays a big role here. Hence, the boundary integral $\mathcal{I}_b$ now can be estimated by using $(H^{-1/2}, H^{1/2})$ dual estimate and the tangential energy estimates can be finished. Doing the divergence and curl estimates is somehow standard and combining with the tangential energy estimates, we can close the a priori estimates.

After we obtaining the a priori estimates, we use linearization method to construct approximate system to \eqref{eq:mhd}. Again, we have two linearized sub-systems: \eqref{sub21} and \eqref{sub22}. Thanks to the boundary smoothing effect of non-collinearity condition \eqref{stc}, we can avoid losing derivatives on the boundary estimates in the linearization iteration, which means the approximate system is asymptotically consistent with the a priori estimates for the original system. Then by a contraction argument, the solutions to \eqref{eq:mhd} can be obtained based on the approximated solutions to the linearized system.  What now remains in the proof of the local well-posedness of \eqref{eq:mhd} is to constructing solutions to the linearized approximate problem \eqref{sub21} and \eqref{sub22}. This solvability can be obtained by the viscosity vanishing method used in \cite[Section 5.1]{GW_16}. Consequently, the construction of solutions to the incompressible MHD equations \eqref{eq:mhd} is completed.
\section{Preliminary}
{\subsection{Notation}\label{Notation}
Einstein's summation convention is used throughout the paper, and repeated
Latin indices $i,j,$ etc., are summed from 1 to 2.
We use $C$ to denote generic constants, which only depends on the domain $\Omega$ and the boundary $\Gamma$, and  use $f\ls g$ to denote $f\leq Cg$.  We use $P$ to denote a generic polynomial function of its arguments, and the polynomial coefficients are generic constants $C$.
We use $D$ to denote the spatial derives: $\partial_r, \bp$.


\subsubsection{Sobolev spaces}
For integers $k\geq0$, we define the axially symmetic Sobolev space $H^k_{r,z}(\Omega)$ to be the completion of the functions in $C^{\infty}(\bar{\Omega})$ in the norm
\begin{equation*}
	\|u\|_k:=\left(\sum_{|\alpha|\leq k}2\pi\int_{\mathbb T}\int_{0}^{R_0}r\abs{D^{\alpha}u(r,z)}^2\,drdz\right)^{1/2}
\end{equation*}
for a multi-index $\alpha\in \mathbb{Z}_{+}^2$. For real numbers $s\geq0$, the Sobolev spaces $H^s_{r,z}(\Omega)$ are defined by interpolation.

\noindent
On the boundary $\Gamma$, for functions $w\in H^k(\Gamma)$, $k\geq0$, we set
\begin{equation*}
	|w|_k:=\left(\sum_{\beta\leq k}2\pi R_0\int_{\mathbb T} \abs{\partial_z^{\beta}w(R_0,z)}^2\,dz\right)^{1/2}
\end{equation*}
for a multi-index $\beta\in\mathbb{Z}_+$. The real number $s\geq0$ Sobolev space $H^s(\Gamma)$ is defined by interpolation. The negative-order Sobolev spaces $H^{-s}(\Gamma)$ are defined via duality: for real $s\geq0$, $H^{-s}(\Gamma):=[H^s(\Gamma)]'$.
\subsection{Product and commutator estimates}

We recall the following product and commutator estimates.
\begin{lemma}
It holds that

\noindent $(i)$ For $|\alpha|=k\geq 0$,
\begin{equation}
\label{co0}
\norm{D^{\alpha}(gh)}_0 \ls \norm{g}_{k}\norm{h}_{[\frac{k}{2}]+2}+\norm{g}_{[\frac{k}{2}]+2}\norm{h}_{k}.
\end{equation}
$(ii)$ For $|\alpha|=k\geq 1$, we define the commutator
\begin{equation*}
[D^{\alpha}, g]h =D^{\alpha}(gh)-gD^{\alpha} h.
\end{equation*}
Then we have
\begin{align}
\label{co1}
&\norm{[D^{\alpha}, g]h}_0\ls\norm{D g}_{k-1}\norm{h}_{[\frac{k-1}{2}]+2}+\norm{D g}_{[\frac{k-1}{2}]+2}\norm{h}_{k-1}.
\end{align}
$(iii)$ For $|\alpha|=k\geq 2$, we define the symmetric commutator
\begin{equation*}
\left[D^{\alpha}, g, h\right] = D^{\alpha}(gh)-D^{\alpha}g h-gD^{\alpha} h.
\end{equation*}
Then we have
\begin{equation}
\label{co2}
\norm{\left[D^{\alpha}, g, h\right]}_0\ls\norm{D g}_{k-2}\norm{D h}_{[\frac{k-2}{2}]+2}+ \norm{D g}_{[\frac{k-2}{2}]+2}\norm{D h}_{k-2} .
\end{equation}
\end{lemma}
\begin{proof}
The proof of these estimates is standard; we first use the Leibniz formula to expand these terms as sums of products and then control the $L^2_{r,z}$ norm of each product with the lower order derivative term in $L^\infty\subset H^2_{r,z} $ and the higher order derivative term in $L^2_{r,z}$. See for instance Lemma A.1 of \cite{Wang_15}.
\end{proof}

We will also use the following lemma.
\begin{lemma}
It holds that
\begin{equation}
\label{co123}
\abs{gh}_{1/2} \ls \abs{g}_{W^{1,\infty}}\abs{h}_{1/2}.
\end{equation}
 \end{lemma}
\begin{proof}
It is direct to check that $\abs{gh}_{s} \ls \abs{g}_{W^{1,\infty}}\abs{h}_{s}$ for $s=0,1$. Then the estimate \eqref{co123} follows by the interpolation.
\end{proof}
\subsection{Hardy-type inequality}
We recall the following Hardy inequality:
\begin{lemma}[A higher order Hardy-type inequality]
\label{hd1}
  Let $s\geq 1$ be a given integer, and suppose that $g\in H^s_{r,z}(\Omega)$
and $g(0,z)=0$, we have
  \begin{equation}
	 \norm{\dfrac{g}{r}}_{s-1}\leq C\norm{g}_s.
    \label{}
  \end{equation}
  \label{hardy}
\end{lemma}
Lemma \ref{hd1} can be proved by a similar approach used in \cite[Lemma 3.1]{Gu_2011}.
\subsection{Geometry Identities}
\begin{align}
\label{dJ}
&\partial J=\dfrac{\partial J}{\partial {\mathcal F}_{ij}}\partial {\mathcal F}_{ij} =  J{\mathcal{A}}_{ij}\partial {\mathcal F}_{ij},\\
&\partial {\mathcal{A}}_{ij}  = -{\mathcal{A}}_{i\ell}\partial {\mathcal F}_{m\ell}{\mathcal{A}}_{mj},
\label{partialF}
\end{align}
where $\partial$ can be $\partial_r$, $\partial_z$ and $\partial_t$ operators.

From the incompressible constraint, we have $\partial_t J= -J \dfrac{v^r}{R}$ for $J=\text{det} \mathcal{F}$, which means $J=\dfrac{r}{R}$.

Moreover, we have the Piola identity
\begin{equation}\label{polia}
	\partial_r\left(J\a_{1j}\right)+\partial_z\left(J\a_{2j}\right) =0.
\end{equation}
\section{Linearized approximate system}\label{lap}
In this section, we construct the approximate system by linearizing method and then derive a priori estimates for this system and also prove the solvability of the system.
\subsection{Approximate system}First, we denote $\zeta=(R, Z), \nu=(v^r, v^z)$ and reformulate the the system \eqref{eq:mhd} into two coupled sub-systems which are both defined in $\Omega$:
\begin{equation}
	\begin{cases}
		\partial_t \zeta= \nu\,\,&\text{in}\,\,\Omega,\\
\partial_t \nu+\nabla_{\a} q-(b_0\cdot\nabla)^2\zeta=\left(\dfrac{(v^{\theta})^2}{R}-R(b_0\cdot\nabla \Theta)^2, 0\right)\,\,&\text{in}\,\,\Omega,\\
\Div_{\a}\nu=0\,\,&\text{in}\,\,\Omega,\\
q=\dfrac{1}{2}\dfrac{C^2(t)}{R^2}\,\,&\text{on}\,\,\Gamma,\\
(\zeta, \nu)|_{t=0}=(r, z, v_0^r, v_0^z).\,\,&\text{in}\,\,\Omega.
\end{cases}
	\label{sub1}
\end{equation}
and
\begin{equation}
	\begin{cases}
		\partial_t \Theta=\dfrac{v^{\theta}}{R}\,\,&\text{in}\,\,\Omega,	\\
	\partial_t v^{\theta}-(b_0\cdot\nabla)(Rb_0\cdot\nabla\Theta)=-\dfrac{v^{\theta}v^r}{R}
	+b_0\cdot\nabla Rb_0\cdot\nabla\Theta\,\,&\text{in}\,\,\Omega,\\(\Theta, v^{\theta})|_{t=0}=(\theta, u_0^{\theta}).
\end{cases}
	\label{sub2}
\end{equation}
where $\nabla_{\a}=(\partial_R^{\a}, \partial_Z^{\a}), \partial_{\zeta_i}^{\a}:=\a_{ij}\partial_{a_j}, \Div_{\a}g=\dfrac{1}{R}\partial_{\zeta_i}^{\a}(Rg_i).$

Then given $\bar\nu=(\bar v^r, \bar v^z), \bar v^{\theta}$, $\bar\zeta=(\bar R, \bar Z), \bar \Theta$,  $\bar{\mathcal F}=\mathcal F(\bar\zeta)$ and $\bar \a = \a(\bar\zeta)$, $\bar C(t)=C(\bar \nu, \bar\zeta),$ we introduce our approximate system as two-step linearized system: firstly, we solve
\begin{equation}
	\begin{cases}
			\partial_t \zeta= \nu\,\,&\text{in}\,\,\Omega,\\
\partial_t \nu+\nabla_{\bar\a} q-(b_0\cdot\nabla)^2\zeta=\left(\dfrac{(\bar v^{\theta})^2}{R}-\bar R(b_0\cdot\nabla \bar\Theta)^2, 0\right)\,\,&\text{in}\,\,\Omega,\\
\Div_{\bar\a}\nu=0\,\,&\text{in}\,\,\Omega,\\
q=\dfrac{1}{2}\dfrac{\bar C^2(t)}{\bar R^2}\,\,&\text{on}\,\,\Gamma,\\
(\zeta, \nu)|_{t=0}=(r, z, v_0^r, v_0^z).\,\,&\text{in}\,\,\Omega.
\end{cases}
	\label{sub21}
\end{equation}
After solving the system \eqref{sub21}, then we solve the following system:
\begin{equation}
	\begin{cases}
		\partial_t \Theta=\dfrac{v^{\theta}}{R}\,\,&\text{in}\,\,\Omega,\\
	\partial_t v^{\theta}-(b_0\cdot\nabla)(Rb_0\cdot\nabla\Theta)=-\dfrac{\bar v^{\theta}\bar v^r}{\bar R}
	+b_0\cdot\nabla \bar Rb_0\cdot\nabla\bar\Theta\,\,&\text{in}\,\,\Omega,\\(\Theta, v^{\theta})|_{t=0}=(\theta, v_0^{\theta}).
\end{cases}
	\label{sub22}
\end{equation}

We define
\begin{equation}\label{apbound}
	M=\sup_{t\in [0,T]}\norm{(\bar v^r,\bar v^{\theta}, \bar v^z,  \bar R, \bar Z, \partial_t\bar R, \partial_t\bar Z, b_0\cdot\nabla \bar R,\bar R b_0\cdot\nabla \bar\Theta, b_0\cdot\nabla \bar Z)}_4^2.
\end{equation}
We take the time $T>0$  sufficiently small so that for $t\in[0,T]$,
\begin{align}
\label{inin3}\abs{\bar{\mathcal F}_{ij}(t)-\delta_{ij}}\leq \dfrac{1}{8}, \abs{\bar\a_{ij}(t)-\delta_{ij}}\leq \dfrac{1}{8} \text{ in }\Omega.
\end{align}
From the definition of $\bar C(t)$:
$$\bar C(t)=C(0)e^{\int_0^t \bar A(\tau)\,d\tau}, \bar A(t)=\dfrac{\int_{\mathbb T}\left(\bar v^r\partial_z\bar Z(R_0, z, t)-\bar v^z\partial_z \bar R(R_0,z,t)\right)\,dz}{\int_{\mathbb T}\left(\ln R_S-\ln \bar R(R_0,z,t)\right)\partial_z \bar Z(R_0, z, t)\,dz},$$
we can also have
\begin{equation}
\sup_{t\in [0,T]}\abs{\bar C(t)} \leq C(M).
\end{equation}
We define the high order energy functional:
\begin{equation}
	\mathfrak{E}(t)=\norm{(v^r, v^{\theta},  v^z, R, Z, b_0\cdot\nabla R, Rb_0\cdot\nabla \Theta, b_0\cdot\nabla Z)}_4^2
\end{equation}
We will prove that $\mathfrak{E}$ remains bounded on a time interval dependent of $M$, which is stated as the following theorem.

\begin{theorem} \label{th43}
There exists a time $T_1$ dependent of $M$ such that
\begin{equation}
\label{bound}
\sup_{[0,T_1]}\mathfrak{E}(t)\leq 2M_0,
\end{equation}
where $M_0=P\left(\norm{\left(v_0^r, v_0^{\theta}, v_0^z\right)}_4^2+\norm{\left(b_0^r, b_0^{\theta}, b_0^z\right)}_4^2\right).$
\end{theorem}
\subsection{A priori estimates for system \eqref{sub21}.}\label{apest}
For system \eqref{sub21}, we have the following a priori estimates:
\begin{proposition} \label{th412}
For $t\in[0,T]$, it holds that:
	\begin{equation}
	\norm{\nu}_4^2+\norm{\zeta}_4^2+\norm{b_0\cdot\nabla\zeta}_4^2\leq M_0+TC(M)P\left(\sup_{t\in [0,T]}\mathfrak E(t)\right).
\end{equation}
\end{proposition}

\subsubsection{Pressure estimates}\label{pressure1}
\begin{proposition}
The following estimate holds:
	\begin{equation}
\label{nueq}
	\norm{q}_4^2\leq C(M)P\left(\sup_{t\in [0,T]}\mathfrak E(t)\right).
\end{equation}
\label{pr}
\end{proposition}
\begin{proof}
Taking $\bar J\Div_{\ak}$ on the second equation of \eqref{sub21} to get:
\begin{equation}
	\begin{cases}
		\dfrac{1}{\bar R}\partial_{a_i}(\bar R\bar E_{ij}\partial_{a_j} q)=G_1+b_0\cdot\nabla G_2 \\
	q|_{\Gamma}=\dfrac{\bar C^2(t)}{\bar R^2}
\end{cases}
	\label{pressure}
\end{equation}
where
\begin{equation*}
	\begin{split}
		&\bar E_{ij}:=\bar J\bar\a_{\ell i}\bar\a_{\ell j},\\
		&G_1:=\bar J\partial_t \bar\a_{i\ell}\partial_{a_{\ell}}\nu_i-\bar J\dfrac{v^r\partial_t \bar R}{\bar R^2}+\left(\dfrac{\bar J}{\bar R}+\bar J\pa^{\bar \a}_R\right)\left(\dfrac{(\bar v^{\theta})^2}{\bar R}-\bar R(b_0\cdot\nabla\bar\Theta)^2\right)+\left[\bar J\Div_{\ak}, b_0\cdot\nabla\right]b_0\cdot\nabla\zeta\\
	&G_2:=\bar J\Div_{\ak}(b_0\cdot\nabla\zeta).
\end{split}
\end{equation*}
Note that  by \eqref{inin3} the matrix $\bar E$ is symmetric and positive.

We denote $\hat h(r,z,t)$ as the harmonic extension of $\frac{\bar C^2(t)}{\bar R^2}$:
\begin{equation*}
\begin{cases}
(\partial_r^2+\dfrac{1}{r}\partial_r+\partial_z^2)\hat h =0\,\,&\text{in } \Omega,\\
\hat h= \frac{\bar C^2(t)}{\bar R^2}\,\,&\text{on } \Gamma,
\end{cases}
\end{equation*}
and by the Trace theorem, we have
\begin{equation}
\label{hq}
	\norm{\hat h}_4^2\ls \abs{\dfrac{\bar C^2(t)}{\bar R^2(R_0,z)}}_{3.5}^2\leq C(M).
\end{equation}
And then $\hat q = q-\hat h$ satisfying the following elliptic equation with zero Dirichlet boundary condition:
\begin{equation}
	\dfrac{1}{\bar R}\partial_{a_i}(\bar R\bar E_{ij}\partial_{a_j}\hat q)=G_1+b_0\cdot\nabla G_2+	\dfrac{1}{\bar R}\partial_{a_i}(\bar R\bar E_{ij}\partial_{a_j} \hat h)
	\label{pressure2}
\end{equation}
Timing $\frac{\bar R}{r}\hat q$ on the equation \eqref{pressure2}, integrating on $\Omega$ and using integration-by-parts, we have
\begin{equation*}
	\int_{\Omega}\dfrac{\bar R}{r}\bar E_{ij}\partial_{a_j} \hat q\partial_{a_{i}} \hat q\,dx=\int_{\Omega}\dfrac{\bar R}{r}G_1\hat q\,dx+\int_{\Omega} \dfrac{\bar R}{r}G_2 b_0\cdot\nabla\hat q\,dx+\int_{\Omega}\dfrac{\bar R}{r}\bar E_{ij}\partial_{a_j} \hat h\partial_{a_i}\hat q\,dx	
\end{equation*}
Thus, with a priori assumption \eqref{inin3} and Poincare's inequality, we arrive at
\begin{equation*}
	\norm{D \hat q}_0^2\ls \norm{G_1}_0^2+\norm{G_2}_0^2+\norm{\bar E_{ij}\partial_{a_j} \hat h}_0^2\leq C(M)P\left(\sup_{t\in [0,T]}\mathfrak{E}(t)\right)
\end{equation*}
and hence with \eqref{hq}, we have
\begin{equation}
\norm{q}_1^2 \leq C(M)P\left(\sup_{t\in [0,T]}\mathfrak{E}(t)\right)
	\label{dfddd}
\end{equation}
Next, applying $\partial_z^{k}$, $k=1,2,3$ to the equation \eqref{pressure2} leads to
\begin{equation*}
\begin{split}
	\dfrac{1}{\bar R}\partial_{a_i}(\bar R\bar E_{ij}\partial_{a_j} \partial_z^k\hat q)=&\partial_z^kG_1+b_0\cdot\nabla \partial_z^kG_2+	\dfrac{1}{\bar R}\partial_{a_i}(\bar R\bar E_{ij}\partial_{a_j} \partial_z^k\hat h)\\&+\dfrac{1}{\bar R}\partial_{a_i}\left(\left[\partial_z^k,\bar R\bar E_{ij}\partial_{a_j}\right](\hat h-\hat q)\right)+\left[\partial_z^k, b_0\cdot\nabla\right]G_2.
\end{split}
\end{equation*}
Thus, similarly, we obtain
\begin{equation*}
\begin{split}
\norm{\partial_z^k \hat q}_1^2\ls& \norm{\partial_z^k G_1}_0^2+\norm{\partial_z^k G_2}_0^2+\norm{\bar E_{ij}\partial_{a_j}\partial_z^k \hat h}_0^2+\norm{\left[\partial_z^k,\bar R\bar E_{ij}\partial_{a_j}\right](\hat h-\hat q)}_0^2\\&+\norm{[\partial_z^k, b_0\cdot\nabla]G_2}_0^2,
\end{split}
\end{equation*}
and then
\begin{equation}
	\label{prs0}
\norm{\partial_z^k \hat q}_1^2\leq C(M)\left(P\left(\sup_{t\in [0,T]}\mathfrak{E}(t)\right)+\norm{\partial_z^{k-1}\hat q}_1^2\right).
\end{equation}
Combining with \eqref{hq} again, we have
\begin{equation}
\label{prs}
\norm{\partial_z^k q}_1^2\leq C(M)\left(P\left(\sup_{t\in [0,T]}\mathfrak{E}^{\kappa}(t)\right)+\norm{\partial_z^{k-1} q}_1^2\right).
\end{equation}
In order to obtain other high order derivatives of $q$, we denote $\mathfrak g=\dfrac{\bar E_{1j}\partial_{a_j} q}{\bar R}$ and rewrite the first equation of \eqref{pressure} as
\begin{equation}
	\label{pressure3}
	\dfrac{1}{\bar R}\partial_r(\bar R^2 \mathfrak g)=\bar R\partial_r\mathfrak g+2\partial_r\bar R\mathfrak g=\mathfrak G
\end{equation}
where
\begin{equation}
\mathfrak G:=\left(G_1+b_0\cdot\nabla G_2+\dfrac{1}{\bar R}\partial_z(\bar R\bar E_{2j}\partial_{a_j} q)\right).
\end{equation}
Then we obtain
\begin{equation}
	\norm{\bar R\partial_r\mathfrak g+2\partial_r\bar R\mathfrak g}_0^2\leq  \norm{\mathfrak G}_0^2\leq C(M)P\left(\sup_{t\in [0,T]}\mathfrak{E}(t)\right).
	\label{jhieoe}
\end{equation}
With integration-by-parts and a priori assumption \eqref{apbound}, \eqref{inin3},  we have
\begin{equation}
\begin{split}
	&\int_{\Omega}\left(\bar R\partial_r\mathfrak g+2\partial_r\bar R\mathfrak g\right)^2\,dx\\=	&\norm{\bar R \partial_r\mathfrak g}_0^2+4\norm{\partial_r\bar R\mathfrak g}_0^2+\int_{\Omega}4\bar R\partial_r\mathfrak g\partial_r\bar R\mathfrak g\,dx\\=&\norm{\bar R \partial_r\mathfrak g}_0^2+2\norm{\partial_r\bar R\mathfrak g}_0^2-2\int_{\mathbb T}\int_0^{R_0}\partial_r(r\partial_r\bar R)\bar R\abs{\mathfrak g}^2\,drdz+2\int_{\mathbb T}R_0\bar R(R_0,z)\abs{\mathfrak g}^2(R_0, z)\,dz\\\geq &\norm{\bar R \partial_r\mathfrak g}_0^2+2\norm{\partial_r\bar R\mathfrak g}_0^2-CT\sup_{t\in [0,T] }\norm{\partial_r(r\partial_r \partial_t \bar R)}_{L^{\infty}}\norm{\mathfrak g}_0^2\\\geq & \norm{\bar R \partial_r\mathfrak g}_0^2+\norm{\mathfrak g}_0^2
\end{split}
	\label{labelss}
\end{equation}
by taking $T$ sufficiently small (only depend on $M$).
Thus,  we arrive at
\begin{equation}
	\norm{\bar R \partial_r\mathfrak g}_0^2+\norm{\mathfrak g}_0^2\leq C(M)P\left(\sup_{t\in [0,T]}\mathfrak{E}(t)\right)
	\label{ef}
\end{equation}
and as a consequence, we have
	\begin{equation}
		\norm{\partial_r (\bar E_{1j}\partial_{a_j}q)}_0^2
\leq \norm{\partial_r\bar R\mathfrak g}_0^2+\norm{\bar R\partial_r\mathfrak g}_0^2\leq C(M)P\left(\sup_{t\in [0,T]}\mathfrak{E}(t)\right).
		\label{jifiefjoq}
	\end{equation}
Then by using \eqref{prs} and a priori assumption \eqref{inin3} again, we have
\begin{equation}
\label{riep}
\norm{\partial^2_r q}_0^2\leq \norm{\dfrac{1}{\bar E_{11}}\partial_r (\bar E_{1j}\partial_{a_j} q)}_0^2+\norm{\dfrac{1}{\bar E_{11}}\partial_r \bar E_{11}\partial_r q}_0^2+\norm{\dfrac{1}{\bar E_{11}}\partial_r (\bar E_{12}\partial_z q)}_0^2 \leq C(M)P\left(\sup_{t\in [0,T]}\mathfrak{E}(t)\right).
\end{equation}
Next, by acting $\partial_z, \partial_z^2$ on the equation \eqref{pressure3}, we have
\begin{equation*}
\begin{split}
\bar R\partial_r\partial_z\mathfrak g+2\partial_r\bar R\partial_z\mathfrak g=&-\partial_z\bar R\partial_r\mathfrak g-2\partial_z\partial_r\bar R\mathfrak g+\partial_z \mathfrak G\\
\bar R\partial_r\partial_z^2\mathfrak g+2\partial_r\bar R\partial_z^2\mathfrak g=&-\partial_z\bar R\partial_r\partial_z\mathfrak g-2\partial_r\partial_z\bar R\partial_z\mathfrak g+\partial_z\left(-\partial_z\bar R\partial_r\mathfrak g-2\partial_z\partial_r\bar R\mathfrak g+\partial_z \mathfrak G\right)
\end{split}
\end{equation*}
Then by a similar approach from \eqref{jhieoe} to \eqref{riep}, for $k=1,2,$ we can obtain
\begin{equation*}
	\norm{\bar R \partial_r\partial_z^k\mathfrak g}_0^2+\norm{\partial_z^k\mathfrak g}_0^2\leq C(M)P\left(\sup_{t\in [0,T]}\mathfrak{E}(t)\right)
\end{equation*}
and hence
\begin{equation*}
\norm{\partial_r^2\partial_z^2 q}_0^2\leq C(M)P\left(\sup_{t\in [0,T]}\mathfrak{E}(t)\right).
\end{equation*}
Finally, we act $\partial_r, \partial_r^2$ on the equation \eqref{pressure3} to obtain for $k=1,2,$
\begin{equation}
\norm{\bar R \partial_r^{k+1}\mathfrak g}_0^2+\norm{\partial_r^k\mathfrak g}_0^2\leq C(M)P\left(\sup_{t\in [0,T]}\mathfrak{E}(t)\right)
	\label{ipre}
\end{equation}
and hence
\begin{equation}
\norm{\partial_r^4 q}_0^2+\norm{\partial_r^3\partial_z q}_0^2\leq C(M)P\left(\sup_{t\in [0,T]}\mathfrak{E}(t)\right).
\end{equation}
Combining with \eqref{prs}, we prove the proposition.
\end{proof}
\subsubsection{Tangential estimates for $\nu=(v^r, v^z)$}
We start with the basic $L^2$ energy estimates.
\begin{proposition}\label{basic}
For $t\in [0,T]$, it holds that
\begin{equation}\label{00estimate}
\norm{\nu(t)}_0^2+\norm{(b_0\cdot\nabla\zeta)(t)}_0^2\leq M_0+TC(M)P\left(\sup_{t\in[0,T]}\mathfrak{E}(t)\right).
\end{equation}
\end{proposition}
\begin{proof}
Taking the $L^2(\Omega)$ inner product of the second equation in \eqref{sub1} with $\nu$ yields
\begin{equation}\label{hhl1}
 \dfrac{1}{2}\dfrac{d}{dt}\int_{\Omega}\abs{\nu}^2 +\int_{\Omega} \nabla_{\bar\a} q\cdot \nu -\int_{\Omega}(b_0\cdot\nabla)^2\zeta \cdot \nu =\int_{\Omega}\left(\dfrac{(\bar v^{\theta})^2}{\bar R}-\bar R(b_0\cdot\nabla \bar\Theta)^2\right) v^r.
\end{equation}
Using the pressure estimates \eqref{nueq}, we have
\begin{equation}
-\int_{\Omega} \nabla_{\bar\a} q\cdot \nu\le\norm{D \a^\kappa}_{L^{\infty}}\norm{\nu}_0\norm{q}_1\le C(M)P\left(\sup_{t\in[0,T]}\mathfrak{E}(t)\right).
\end{equation}
By Hardy's inequality, we have
\begin{equation}
\int_{\Omega}\left(\dfrac{(\bar v^{\theta})^2}{\bar R}-\bar R(b_0\cdot\nabla\bar\Theta)^2\right) v^r \le C(M)P\left(\sup_{t\in[0,T]}\mathfrak{E}(t)\right).
\end{equation}
Since $b_0$ satisfies \eqref{bcond}, by the integration by parts, we obtain
\begin{equation}\label{hhl3}
\begin{split}
-\int_{\Omega}(b_0\cdot\nabla)^2\zeta\cdot \nu&=\int_{\Omega}b_0\cdot\nabla\zeta_ib_0\cdot\nabla \nu_i\\&=\int_{\Omega}b_0\cdot\nabla\zeta_ib_0\cdot\nabla \partial_t\zeta_i\,dx
\\&=\dfrac{1}{2}\dfrac{d}{dt}\int_{\Omega}\abs{b_0\cdot\nabla\zeta}^2\,dx
\end{split}
\end{equation}
Then \eqref{hhl1}--\eqref{hhl3} implies,
\begin{equation}
\dfrac{d}{dt}\int_{\Omega}\abs{\nu}^2+\abs{b_0\cdot\nabla\zeta}^2 \leq C(M)P\left(\sup_{t\in[0,T]}\mathfrak E(t)\right).
\end{equation}
Integrating directly in time of the above yields \eqref{00estimate}.
\end{proof}
In order to perform higher order tangential energy estimates, one needs to compute the equations satisfied by $(\bp^4 \nu, \bp^4 q, \bp^4 \zeta)$, which requires to commutate $\bp^4$ with each term of $\pa^{\bar\a}_{\zeta_i}$. It is thus useful to establish the following general expressions and estimates for commutators.
 we have
\begin{equation}
	\bp^4 (\pa^{\bar\a}_{\zeta_i}g) =  \pa^{\bar\a}_{\zeta_i} \bp^4 g + \bp^4 {\bar\a}_{ij} \pa_{a_j} g+\left[\bp^4, {\bar\a}_{ij} ,\pa_{a_j} g\right].
\end{equation}
By the identity \eqref{partialF}, we have that
\begin{equation}
\begin{split}
&\bp^4 ({\bar\a}_{ij} \pa_{a_j} g)=-\bp^3({\bar\a}_{i\ell}\bp\pa_{a_{\ell}}  \zeta_m  {\bar\a}_{mj})\pa_{a_j} g
\\&\quad=-{\bar\a}_{i\ell}\pa_{a_\ell} \bp^4\zeta^m{\bar\a}_{mj}\pa_{a_j} g
-\left[\bp^3, {\bar\a}_{i\ell}{\bar\a}_{mj} \right]\bp \pa_{a_\ell}  \zeta_m \pa_{a_j} g
\\&\quad=-\pa^{\bar\a}_{\zeta_i}( \bp^4\zeta\cdot\nabla_{\bar\a} g)+\bp^4\zeta\cdot\nabla_{\bar\a}( \pa^{\bar\a}_{\zeta_i} g)
	-\left[\bp^3, {\bar\a}_{i\ell}{\bar\a}_{mj}\right]\bp\pa_{a_\ell} \zeta_m\pa_{a_j} g.\end{split}
\end{equation}
It then holds that
\begin{equation}\label{commf}
	\bp^4 (\pa^{\bar\a}_{\zeta_i}g) =  \pa^{\bar\a}_{\zeta_i}\left(\bp^4 g- \bp^4\zeta\cdot\nabla_{\bar\a} g\right)+ \mathcal{C}_i(g).
\end{equation}
where the commutator $\mathcal{C}_i(g)$ is given by
\begin{equation}
	\mathcal{C}_i(g)=\left[\bp^4, {{\bar\a}_{ij}},\pa_{a_j} g\right]-\bp^4\zeta\cdot\nabla_{\bar\a}( \pa^{\bar\a}_{\zeta_i}g)
	+\left[\bp^3, {\bar\a}_{i\ell}\bar\a_{mj}\right]\bp\pa_{a_\ell} \zeta_m\pa_{a_j} g
\end{equation}
It was first observed by Alinhac \cite{Alinhac} that the highest order term of $\zeta$ will be cancelled when one uses the good unknown $\bp^4 g- \bp^4\zeta\cdot\nabla_{\bar\a} g$, which allows one to perform high order energy estimates.

The following lemma deals with the estimates of the commutator $\mathcal C_i(g)$.
\begin{lemma}
The following estimate holds:
\begin{equation}\label{comest}
	\norm{\mathcal C_i (g)}_0\leq  P\left(\norm{(\bar R,\bar Z)}_4\right) \norm{g}_4.
\end{equation}
\end{lemma}
\begin{proof}
First, by the commutator estimates \eqref{co2}, we have
\begin{equation}\label{Calpha1}
	\norm{[\bp^4, \bar\a_{ij}, \partial_{a_j} g]}_0\ls\norm{\bar\a}_{3}\norm{D g}_{3}\leq P\left(\norm{(\bar R, \bar Z)}_4\right) \norm{g}_4.
\end{equation}
Next, we get
\begin{equation}
	\norm{\bp^4\zeta\cdot\nabla_{\bar\a}( \pa^{\bar\a}_{\zeta_i} g)}_0\le \norm{\bp^4\zeta}_0\norm{\nabla_{\bar\a}( \pa^{\bar\a}_{\zeta_i} g)}_{L^{\infty} } \leq P\left(\norm{(\bar R, \bar Z)}_4\right) \norm{g}_4.
\end{equation}
Finally, by the commutator estimates \eqref{co1}, we obtain
\begin{equation}\label{Calpha3}
	\norm{\left[\bp^3, \bar\a_{i\ell}\bar\a_{mj}\right]\bp\pa_{a_\ell}\zeta^m\pa_{a_j} g}_0\le \norm{\left[\bp^3, \bar\a_{i\ell}\bar\a_{mj}\right]\bp\pa_{a_\ell} \zeta_m}_0\norm{Dg}_{L^{\infty} }
	\leq P\left(\norm{(\bar R,\bar Z)}_4\right) \norm{g}_3.
\end{equation}

Consequently, the estimate \eqref{comest} follows by collecting \eqref{Calpha1}--\eqref{Calpha3}.
\end{proof}

We now introduce the good unknowns
\begin{equation}
\label{gun}
\mathcal{V}=\bp^4 \nu- \bp^4\zeta\cdot\nabla_{\bar\a} \nu,\quad \mathcal{Q}=\bp^4 q- \bp^4\zeta\cdot\nabla_{\bar\a} q.
\end{equation}

With the condition \eqref{qbdd}, we have \begin{equation}
	\mathcal{Q}=\dfrac{1}{2}\bp^4\left(\dfrac{\bar C^2(t)}{\bar R^2}\right)-\bp^4\zeta\cdot\nabla_{\bar\a} q\quad\text{on } \Gamma.
	\label{qbdc}
\end{equation}
Applying $\bp^4$ to the second equation of \eqref{sub21}, by \eqref{commf}, one gets
\begin{equation}\label{eqValpha}
	\begin{split}
		\partial_t\mathcal{V} &+ \nabla_{\bar\a}\mathcal{Q}-(b_0\cdot\nabla)\left(\bp^4(b_0\cdot\nabla \zeta)\right)\\&=F:=  \dt\left(\bp^4\zeta\cdot\nabla_{\bar\a} \nu\right) - \mathcal{C}_i(q) +\left[\bp^4, b_0\cdot\nabla\right]b_0\cdot\nabla \zeta+\bp^4\left(\dfrac{(\bar v^{\theta})^2}{\bar R}-\bar R(b_0\cdot\nabla\bar\Theta)^2\right),
\end{split}
\end{equation}
and
\begin{equation} \label{divValpha}
	\Div_{\bar\a}\mathcal{V}=g_3:=- \mathcal{C}_i(v^r)-\mathcal{C}_i(v^z)-\left(\bp^4\left(\dfrac{v^r}{\bar R}\right)-\dfrac{\mathcal{V}_1}{\bar R}\right) .
\end{equation}

We shall now derive the $\bp^4$-energy estimates and have the following proposition
\begin{proposition}\label{te}
For $t\in [0,T]$, it holds that
\begin{equation}
\label{teee}
\norm{\partial_z^4 \nu}_0^2+\norm{\partial_z^4(b_0\cdot\nabla \zeta)(t)}_0^2\leq M_0+C(M)TP\left(\sup_{t\in[0,T]}\mathfrak{E}(t)\right).
\end{equation}
\end{proposition}
\begin{proof}
Taking the $L^2(\Omega)$ inner product of \eqref{eqValpha} with $\mathcal{V}$ yields
\begin{equation}
	\dfrac{1}{2}\dfrac{d}{dt}\int_{\Omega}|\mathcal{V}|^2\,dx+\int_{\Omega}\pa^\a_{\zeta_i}\mathcal Q \mathcal{V}_i\,dx+\int_{\Omega}\bp^4(b_0\cdot\nabla \zeta_i)b_0\cdot\nabla\mathcal{V}_i\,dx=\int_{\Omega}F\cdot\mathcal{V}\,dx.
	\label{fd}
\end{equation}
Firstly, the right hand side of \eqref{fd} can be bounded by
\begin{equation}
\begin{split}
\quad\int_{\Omega}F\cdot\mathcal{V}\,dx &\leq \bigg(\norm{\dt\left(\bp^4\zeta\cdot\nabla_{\bar\a}\nu\right)}_0 - \norm{\mathcal{C}_i(q)}_0+\norm{\left[\bp^4, b_0\cdot\nabla\right]b_0\cdot\nabla\zeta}_0\\&\quad+\norm{\bp^4\left(\dfrac{(\bar v^{\theta})^2}{\bar R}-\bar R(b_0\cdot\nabla\bar\Theta)^2\bigg)}_0\right)\norm{\mathcal V}_0\\
&\leq C(M)P\left(\sup_{t\in [0,T]}\mathfrak{E}(t)\right).
\end{split}
\end{equation}
Here we use \eqref{comest} and we estimate
\begin{equation}
	\begin{split}
	\norm{\bp^4\left(\dfrac{(\bar v^{\theta})^2}{\bar R}\right)}_0&\leq C\left(\norm{\bar R \bp^4(\dfrac{\bar v^{\theta}}{\bar R})}_0\norm{\dfrac{\bar v^{\theta}}{\bar R}}_{L^{\infty}}+\norm{\bar v^{\theta}}_4\norm{\dfrac{\bar v^{\theta}}{\bar R}}_{L^{\infty}}\right)
	\\&\leq C\left(\norm{\bp^4\bar v^{\theta}-\left[\bp^4, \dfrac{\bar v^{\theta}}{\bar R}\right]\bar R}_0 \norm{\dfrac{\bar v^{\theta}}{\bar R}}_{L^{\infty}}+\norm{\bar v^{\theta}}_4\norm{\dfrac{\bar v^{\theta}}{\bar R}}_{L^{\infty}}\right) \leq C(M)
\end{split}
\end{equation}
by using Hardy's inequality \eqref{hardy} and $\norm{\bp^4(\bar R(b_0\cdot\nabla\bar\Theta)^2)}_0$ can also be bounded by $C(M)$.

Next, with \eqref{deflag}, we have
\begin{equation*}
	\begin{split}
		&\int_{\Omega}\bp^4(b_0\cdot\nabla \zeta_i) b_0\cdot\nabla\mathcal{V}_i\,dx\\=&\int_{\Omega}\bp^4(b_0\cdot\nabla \zeta_i)b_0\cdot\nabla(\bp^4\nu_i)\,dx-\int_{\Omega}\bp^4(b_0\cdot\nabla \zeta_i)b_0\cdot\nabla(\bp^4\zeta\cdot\nabla_{\bar\a}\nu_i)\,dx\\=&\dfrac{1}{2}\dfrac{d}{dt}\int_{\Omega}|\bp^4(b_0\cdot\nabla \zeta)|^2\,dx+\int_{\Omega}\bp^4(b_0\cdot\nabla \zeta_i)\left[b_0\cdot\nabla,\bp^4\right]\nu_i\,dx\\&-\int_{\Omega}\bp^4(b_0\cdot\nabla \zeta_i)b_0\cdot\nabla(\bp^4\zeta\cdot\nabla_{\bar\a}\nu_i)\,dx\\\geq&\dfrac{1}{2}\dfrac{d}{dt}\left(2\pi\int_{\mathbb T}\int_{0}^{R_0}r|\bp^4(b_0\cdot\nabla \zeta)|^2\,drdz\right)-C(M)P\left(\sup_{t\in [0,T]}\mathfrak{E}(t)\right).
	\end{split}
\end{equation*}
By integration-by-parts and \eqref{divValpha}, we have
\begin{equation}
	\begin{split}
		\int_{\Omega}\pa^{\bar\a}_{\zeta_i}\mathcal Q \mathcal{V}_i\,dx=&-\int_{\Omega}\dfrac{1}{\bar R}\mathcal{Q}\left(\pa^{\bar\a}_{\zeta_i}(\bar R\mathcal{V}_i)\right)\,dx-\int_{\Omega}\pa_{a_j}\left(\dfrac{r}{\bar R}\bar \a_{ij}\right)\mathcal Q\dfrac{\bar R}{r}\mathcal{V}_i\,dx+2\pi\int_{\mathbb T}R_0\mathcal Q\left(\mathcal{V}_i\bar\a_{i1}\right)\, dz\\\leq& \norm{\mathcal{Q}}_0\left(\norm{g_3}_0+\norm{\mathcal V}_0\norm{\dfrac{r}{\bar R}\bar\a}_3\right)+C(M)\abs{\mathcal{Q}}_{1/2}\abs{\mathcal{V}}_{-1/2}
		\\\leq& \norm{\mathcal{Q}}_0\left(\norm{g_3}_0+\norm{\mathcal V}_0\norm{\dfrac{r}{\bar R}\bar\a}_3\right)+C(M)\abs{\partial_z^4 \left(\dfrac{\bar C^2(t)}{\bar R^2}\right)-\partial_z^4\bar\zeta_i \pa_{\zeta_i}^{\bar\a}q}_{1/2}\abs{\mathcal{V}}_{-1/2}		
\\\leq  & \norm{\mathcal{Q}}_0\left(\norm{g_3}_0+\norm{\mathcal V}_0\norm{\dfrac{r}{\bar R}\bar\a}_3\right)\\&\quad+C(M)\abs{\partial_z^3 \left(\dfrac{\bar C^2(t)b_0^z\partial_z\bar R}{b_0^z}\dfrac{1}{\bar R^3}\right)-\partial_z^3\left(\dfrac{b_0^z\partial_z \bar \zeta_i}{b_0^z}\right)\pa_{\zeta_i}^{\bar\a}q}_{1/2}\abs{\mathcal{V}}_{-1/2}\\\leq & C(M)P\left(\sup_{t\in [0,T]}\mathfrak{E}(t)\right).
\end{split}
	\label{j}
\end{equation}
Here we used the proposition \ref{pr}, the condition $\abs{b_0^z}\geq \delta$ on $\Gamma$ and the trace theorem
\begin{equation*}
\abs{\partial_z^3(b_0^z\partial_z\bar \zeta)}_{1/2}\leq C\norm{b_0\cdot\nabla\bar \zeta}_4
\end{equation*}
to control the last term on the RHS of \eqref{j}. Moreover, we also estimate $\norm{g_3}_0$ as
\begin{equation*}
\begin{split}
\norm{g_3}_0&\leq \norm{C_i(\nu)}_0+\norm{\bp^4\left(\dfrac{v^r}{\bar R}\right)-\dfrac{\bp^4v^r-\bp^4\bar R\partial_{R}^{\bar\a}v^r-\bp^4\bar Z\partial_{R}^{\bar\a}v^r}{\bar R}}_0\\&\leq C(M)P\left(\sup_{t\in [0,T]}\mathfrak{E}(t)\right)+\norm{-\dfrac{v^r\bp^4\bar R}{\bar R^2}+\dfrac{\bp^4\bar R\partial_{R}^{\bar\a}v^r}{\bar R}+\dfrac{\bp^4\bar Z\partial_{Z}^{\bar\a}v^r}{\bar R}}_0
\\&\leq C(M)P\left(\sup_{t\in [0,T]}\mathfrak{E}(t)\right)+\norm{\bp^4\bar R\partial_R^{\bar\a}\left(\dfrac{v^r}{\bar R}\right)}_0+\norm{\bp^4\bar Z}_0\norm{\dfrac{\partial_Z^{\bar\a}v^r}{\bar R}}_{L^\infty}
\\&\leq C(M)P\left(\sup_{t\in [0,T]}\mathfrak{E}(t)\right)
\end{split}
\end{equation*}
by using Hardy's inequality.

Combining \eqref{fd}-\eqref{j}, we arrive at
\begin{equation*}
\sup_{t\in[0,T]}\left(\norm{\mathcal{V}}_0^2+\norm{\partial_z^4(b_0\cdot\nabla\zeta)}_0^2\right) \leq M_0+C(M)TP\left(\sup_{t\in [0,T]}\mathfrak{E}(t)\right).
\end{equation*}
Then by the definition of $\mathcal{V}$ and the first equation of \eqref{sub21}, we prove the proposition.
\end{proof}
\subsection{Normal estimates for $\nu=(v^r, v^z)$}\label{curle}
In this subsection, we control the normal derivatives of $\nu=(v^r, v^z)$ by using the equaton \eqref{divValpha}, and the curl equation
\begin{equation}
	\begin{split}
	&	\partial_t(\pa^{\bar\a}_Z v^r-\pa^{\bar\a}_R v^z)-b_0\cdot\nabla\left(\pa^{\bar\a}_Z(b_0\cdot\nabla R)-\pa^{\bar\a}_R(b_0\cdot\nabla Z)\right)\\&=\left[\pa^{\bar\a}_Z, b_0\cdot\nabla\right](b_0\cdot\nabla R)-\left[\pa^{\bar\a}_R, b_0\cdot\nabla\right](b_0\cdot\nabla Z)+\left[\partial_t,\pa^{\bar\a}_Z\right] v^r-\left[\partial_t,\pa^{\bar\a}_R\right] v^z
\end{split}
	\label{curlequ}
\end{equation}
We denote $\curl \nu:=\partial_z v^r-\partial_r v^z, \curl b:=\partial_z(b_0\cdot\nabla R)-\partial_r(b_0\cdot\nabla Z)$, $\curl_{\bar\a}\nu:=\partial_Z^{\bar\a}v^r-\partial_R^{\bar\a}v^z$, $\curl_{\bar\a} b:=\partial_Z^{\bar\a}(b_0\cdot\nabla R)-\partial_R^{\bar\a}(b_0\cdot\nabla Z)$ and begin with the energy estimates for \eqref{curlequ}.
\begin{proposition}
	\label{curlprop}
For $t\in [0,T]$ with $T\le T_\kappa$, it holds that
\begin{equation}
\label{curlest}
\norm{\curl v(t)}_{3}^2+\norm{\curl b(t)}_{3}^2\leq M_0+C(M)TP\left(\sup_{t\in[0,T]} \mathfrak{E}(t)\right).
\end{equation}
\end{proposition}
\begin{proof}
Apply $D^3$ to \eqref{curlequ} to get
\begin{equation}\label{curl2}
	\partial_t(D^3(\curl_{\bar\a} v))-b_0\cdot\nabla\left(D^3\left(\curl_{\bar\a} b\right)\right)=F,
\end{equation}
with
\begin{equation*}
	\begin{split}
	F:=&[D^3,b_0\cdot\nabla]\left(\curl_{\bar\a} b\right)\\&+D^3\left([\pa^{\bar\a}_Z, b_0\cdot\nabla](b_0\cdot\nabla R)-[\pa^{\bar\a}_R, b_0\cdot\nabla](b_0\cdot\nabla Z)+[\partial_t,\pa^{\bar\a}_Z] v^r-[\partial_t,\pa^{\bar\a}_R] v^z\right).
\end{split}
\end{equation*}
Taking the $L^2$ inner product of \eqref{curl2} with $D^3(\pa^{\bar\a}_Z v^r-\pa^{\bar\a}_R v^z)$, by the integration by parts, we get
\begin{align}\label{j00}
	&\dfrac{1}{2}\dfrac{d}{dt}\int_{\Omega}\abs{D^3(\curl_{\bar\a} v)}^2\,dx+\underbrace{\int_{\Omega}D^3(\curl_{\bar\a}b) D^3(\pa^{\bar\a}_Z (b_0\cdot\nabla v^r)-\pa^{\bar\a}_R(b_0\cdot\nabla v^z))}_{\mathcal{J}_1}\\&\quad=\underbrace{\int_\Omega F D^3(\pa^{\bar\a}_Z v^r-\pa^{\bar\a}_R v^z)}_{\j_2}+\underbrace{\int_{\Omega}D^{3}(\curl_{\bar\a}b)(\left[D^3\pa^{\bar\a}_Z, b_0\cdot\nabla\right] v^r-\left[D^3\pa^{\bar\a}_Z, b_0\cdot\nabla\right]v^z)}_{\j_3}.\nonumber
\end{align}
Recalling \eqref{deflag}, we have
\begin{equation}
\label{j0}
\begin{split}
	\j_1=&\dfrac{1}{2}\dfrac{d}{dt}\int_{\Omega}\abs{D^3\left(\curl_{\bar\a}b\right)}^2\,dx
	-\underbrace{\dfrac{1}{2}\int_{\Omega}D^3\left(\curl_{\bar\a}b\right)D^3\left(\partial_t\bar\a D(b_0\cdot\nabla \zeta)\right)}_{\j_{1a}}.
\end{split}
\end{equation}
By the identity \eqref{partialF}, we have
\begin{equation*}
\j_{1a}
\ls  \norm{\bar\a}_3\norm{D(b_0\cdot\nabla \zeta)}_3\norm{\partial_t\bar\a}_3\norm{D(b_0\cdot\nabla \zeta)}_3
\le C(M)P\left(\sup_{t\in[0,T]} \mathfrak{E}(t)\right).
\end{equation*}
Hence, we obtain
\begin{equation}
\label{j1}
\j_1\ge \dfrac{1}{2}\dfrac{d}{dt}\int_{\Omega}\abs{D^3\left(\curl_{\bar\a}b\right)}^2
-C(M)P\left(\sup_{t\in[0,T]} \mathfrak{E}(t)\right).
\end{equation}
We now turn to estimate the right hand side of \eqref{j00}. By the identity \eqref{partialF}, we may have
\begin{equation}
\label{j2}
\begin{split}
	\j_2\leq  \norm{F}_0 \norm{\curl_{\bar\a}v}_3
	\leq C(M)P\left(\sup_{t\in[0,T]} \mathfrak{E}(t)\right).
\end{split}
\end{equation}
Similarly,
\begin{equation}
	\label{j3}
	\begin{split}
	\j_3&\ls \norm{D^{3}(\curl_{\bar\a}b))}_0\norm{\left[D^3\pa^{\bar\a}_Z, b_0\cdot\nabla\right] v^r-\left[D^3\pa^{\bar\a}_Z, b_0\cdot\nabla\right]v^z}_0\\&
\leq C(M)P\left(\sup_{t\in[0,T]} \mathfrak{E}(t)\right).
\end{split}
\end{equation}
Consequently, plugging the estimates \eqref{j1}--\eqref{j3} into \eqref{j00}, we obtain
\begin{equation} \label{kllj}
	\dfrac{d}{dt}\int_{\Omega}\abs{D^3(\curl_{\bar\a}v)}^2+\abs{D^3\left(\curl_{\bar\a}b\right)}^2\,dx
		\le C(M)P\left(\sup_{t\in[0,T]} \mathfrak{E}(t)\right).
\end{equation}
Integrating \eqref{kllj} directly in time, and applying the fundamental theorem of calculous,
\begin{equation}
	\norm{\curl  f(t)}_3\le \norm{\curl_{\bar\a}f(t)}_3+\norm{\int_0^t\partial_t\bar\a d\tau D f(t)}_3,
\end{equation}
we then conclude the proposition.
\end{proof}
We now derive the divergence estimates. We denote $\Div v:=\partial_rv^r+\dfrac{1}{r}v^r+\partial_zv^z, \Div b:=\partial_r(b_0\cdot\nabla R)+\dfrac{1}{r}(b_0\cdot\nabla R)+\partial_z(b_0\cdot\nabla Z)$.
\begin{proposition}
	\label{divprop}
For $t\in [0,T]$ with $T$, it holds that
\begin{equation}
\label{divest}
\norm{\Div \nu(t)}_3^2+\norm{\Div b(t)}_3^2 \leq C(M)TP\left(\sup_{t\in[0,T]} \mathfrak{E}(t)\right).
\end{equation}
\end{proposition}
\begin{proof}
	From the third equation of \eqref{sub21} and $\bar\a|_{t=0}=I$, we see that
\begin{equation}
	\label{fpe}
	\partial_r v^r+\dfrac{1}{r}v^r+\partial_z v^z=-\int_0^t\partial_t\bar\a_{ij}\,d\tau\partial_{a_j}\nu_i
	+\int_0^t\dfrac{r\partial_t\bar R}{\bar R^2}\,d\tau\dfrac{v^r}{r}.
\end{equation}
Hence, it is clear that by the identity \eqref{partialF},
\begin{equation}\label{divv1}
\norm{\partial_r v^r+\dfrac{1}{r}v^r+\partial_z v^z}_{3}^2\le C(M)TP\left(\sup_{t\in[0,T]} \mathfrak{E}(t)\right).
\end{equation}
From the third equation of \eqref{sub21} again, we have
\begin{equation*}
	\Div_{\bar\a}b_0\cdot\nabla \nu=\left[\Div_{\bar\a},b_0\cdot\nabla\right]\nu.
\end{equation*}
This together with the equation $\nu=\partial_t \zeta$, we have
\begin{equation}
\label{divb}
\begin{split}
	\partial_t\left(\Div_{\bar\a}(b_0\cdot\nabla \zeta)\right)=&\left[\Div_{\bar\a},b_0\cdot\nabla\right]\nu+ \partial_t \bar\a_{i\ell} \partial_{a_\ell}\left(b_0\cdot\nabla \zeta_i\right)+\partial_t\dfrac{1}{\bar R}b_0\cdot\nabla R.
\end{split}
\end{equation}
This implies that, by doing the $D^3$ energy estimate and using the identity \eqref{partialF},
\begin{equation}\label{di2}
\norm{\Div_{\bar\a}(b_0\cdot\nabla \zeta)}_{3}^2
\leq  C(M)TP\left(\sup_{t\in[0,T]} \mathfrak{E}(t)\right).
\end{equation}
And then applying the fundamental theorem of calculous, and a similar equality as \eqref{fpe}, we arrive at
\begin{equation}\label{divv12}
\norm{\Div b}_{3}^2
\leq  C(M)TP\left(\sup_{t\in[0,T]} \mathfrak{E}(t)\right).
\end{equation}

Consequently, we conclude the proposition by the estimates \eqref{divv1} and \eqref{divv12}.
\end{proof}
Now we show how to get normal derivatives of $\nu=(v^r, v^z)$ and $b_0\cdot\nabla\zeta$ by using Proposition \ref{curlprop} and Proposition \ref{divprop}:
\begin{proposition}\label{nore}
For $t\in [0,T]$, it holds that
\begin{equation}
\label{teedde}
\norm{D^4 \nu (t)}_0^2+\norm{D^4(b_0\cdot\nabla \zeta)(t)}_0^2\leq M_0+C(M)TP\left(\sup_{t\in[0,T]}\mathfrak{E}(t)\right).
\end{equation}
\end{proposition}
\begin{proof}
First, by using the Proposition \ref{curlprop} and the Proposition \ref{te}, we have
\begin{equation}
\label{fefe}
\begin{split}
&\norm{\partial_z^3\partial_rv^z}_0^2\leq M_0+C(M)TP\left(\sup_{t\in[0,T]} \mathfrak{E}(t)\right),\\
&\norm{\partial_z^3(\partial_rv^r+\dfrac{1}{r}v^r)}_0^2\leq M_0+C(M)TP\left(\sup_{t\in[0,T]} \mathfrak{E}(t)\right).
\end{split}
\end{equation}
By direct calculation, we have
\begin{equation}
	\norm{\partial_z^3\partial_r v^r+\partial_z^3(\dfrac{v^r}{r})}_0^2=\norm{\partial_z^3\partial_r v^r}_0^2+\norm{\partial_z^3(\dfrac{v^r}{r})}_0^2+2\int_{\Omega}\partial_z^3\partial_r(r\dfrac{v^r}{r})\partial_z^3(\dfrac{v^r}{r})\,dx
	\label{sln}
\end{equation}
and by integration-by-parts, the last term of the above equality can be calculated as
\begin{equation}	2\int_{\Omega}\partial_z^3\partial_r(r\dfrac{v^r}{r})\partial_z^3(\dfrac{v^r}{r})\,dx=\int_{\Omega}|\partial_z^3(\dfrac{v^r}{r})|^2\,dx+\int_{\mathbb T}\abs{\partial_z^3(\dfrac{v^r}{r})}^2(R_0,z)R_0^2\,dz
	\label{sfo}
\end{equation}
Thus, we arrive at
\begin{equation}
	\norm{\partial_z^3\partial_rv^r}_0^2+\norm{\partial_z^3(\dfrac{v^r}{r})}_0^2\leq M_0+C(M)TP\left(\sup_{t\in[0,T]} \mathfrak{E}(t)\right).
	\label{divesty}
\end{equation}
Next,  we use the Proposition \ref{curlprop} and \eqref{divesty} to obtain
\begin{equation}
\norm{\partial_z^2\partial_r^2v^z}_0^2\leq M_0+C(M)TP\left(\sup_{t\in[0,T]} \mathfrak{E}(t)\right),
\end{equation}
and use the Proposition \ref{divprop} and \eqref{fefe} to obtain
\begin{equation*}
\norm{\partial_z^2\partial_r(\partial_r v^r+\dfrac{1}{r}v^r)}_0^2\leq M_0+C(M)TP\left(\sup_{t\in[0,T]} \mathfrak{E}(t)\right),
\end{equation*}
By direct calculation, we have
\begin{equation}
	\norm{\partial_z^2\partial_r^2 v^r+\partial_z^2\partial_r(\dfrac{v^r}{r})}_0^2=\norm{\partial_z^2\partial_r^2 v^r}_0^2+\norm{\partial_z^2\partial_r(\dfrac{v^r}{r})}_0^2+2\int_{\Omega}\partial_z^2\partial_r^2(r\dfrac{v^r}{r})\partial_z^2\partial_r(\dfrac{v^r}{r})\,dx
	\label{sln2}
\end{equation}
and by integration-by-parts, the last term of the above equality can be calculated as
\begin{equation}	2\int_{\Omega}\partial_z^2\partial_r^2(r\dfrac{v^r}{r})\partial_z^2\partial_r(\dfrac{v^r}{r})\,dx=3\int_{\Omega}|\partial_z^2\partial_r(\dfrac{v^r}{r})|^2\,dx+\int_{\mathbb T}\abs{\partial_z^2\partial_r(\dfrac{v^r}{r})}^2(R_0,z)R_0^2\,dz.
	\label{sfo2}
\end{equation}
Thus, we arrive at
\begin{equation}
	\norm{\partial_z^2\partial_r^2v^r}_0^2+\norm{\partial_z^2\partial_r(\dfrac{v^r}{r})}_0^2\leq M_0+C(M)TP\left(\sup_{t\in[0,T]} \mathfrak{E}(t)\right).
	\label{divesty2}
\end{equation}
Last, we can repeat the above process inductively to bound the $\norm{\partial_z\partial_r^3(v^r,v^z)}_0^2+\norm{\partial_r^4(v^r,v^z)}_0^2$ by
$M_0+C(M)TP\left(\sup\limits_{t\in[0,T]} \mathfrak{E}(t)\right)$.

The estimates for $b_0\cdot\nabla R, b_0\cdot\nabla Z$ can be obtained by a similar way and the proposition is proved.
\end{proof}
Combining the Proposition \ref{basic} and the Proposition \ref{nore}, we arrive at the Proposition \ref{th412}.
\subsection{Solvability of the system \eqref{sub21}}\label{solv21}
With the above a priori estimates, we can solve the system \eqref{sub21} by applying the artificial viscosity approach used in \cite[Section 5.1]{GW_16}. Here we just give a sketch and the reader can refer to \cite{GW_16} for the details.

Firstly, we construct the linear $\varepsilon$-approximate problem by adding artificial viscosity:
 \begin{equation}
 \label{epsapproximate}
 	\begin{cases}
		\partial_t \zeta= \nu+\epsilon (b_0\cdot\nabla)^2 \zeta\,\,&\text{in}\,\,\Omega,\\
\partial_t \nu+\nabla_{\bar\a} q-(b_0\cdot\nabla)^2\zeta=
\left(\dfrac{(\bar v^{\theta})^2}{\bar R}-\bar R(b_0\cdot\nabla \bar\Theta)^2,0\right)\,\,&\text{in}\,\,\Omega,\\
\Div_{\bar\a}\nu=0\,\,&\text{in}\,\,\Omega,\\
q=\dfrac{1}{2}\dfrac{\bar C^2(t)}{\bar R^2}\,\,&\text{on}\,\,\Gamma,\\
(\zeta, \nu)|_{t=0}=(r, z, v_0^r, v_0^z).\,\,&\text{in}\,\,\Omega.
\end{cases}
 \end{equation}

Secondly,  we solve this artificial viscosity system \eqref{epsapproximate} by using a fixed point argument which is based on the solvability of three linear problems, i.e., \eqref{eqforeta}, \eqref{eqforv} and \eqref{eqforq}.
\subsubsection{Three linear problems}
The first linear problem is the following linear degenerate parabolic problem of $\zeta$ with given $\mathfrak{f}^1$:
\begin{equation}
\label{eqforeta}
\begin{cases}
\partial_t\zeta-\epsilon(b_0\cdot\nabla)^2\zeta=\mathfrak{f}^1 \,\, &\text{in } \Omega,\\
\zeta|_{t=0}=(r,z).
\end{cases}
\end{equation}
\begin{proposition}\label{f1}
Given $\mathfrak{f}^1 \in L^{2}(0,T; H_{r,z}^4(\Omega))$, then the problem \eqref{eqforeta}
admits a unique solution $(R, Z)$ that achieves
the initial data and satisfies
\begin{equation}\label{b1}
\begin{split}
\norm{  \zeta(t)}_4+ \varepsilon\norm{ b_0\cdot\nabla\zeta(t) }_4^2 + &\varepsilon^2\int_0^t\norm{ ( b_0\cdot\nabla)^2\zeta }_4^2d\tau
\\&\ls  \norm{(r, z)}_4+ \varepsilon\norm{ b_0\cdot\nabla(r, z) }_4^2  +\int_0^t\norm{\mathfrak{f}^1}_4^2d\tau.
\end{split}
\end{equation}
\end{proposition}

The second linear problem is the simple transport problem of $v$ with given $\mathfrak{f}^2$:
\begin{equation}
\label{eqforv}
\begin{cases}
	\partial_t \nu =\mathfrak{f}^2 &\text{in } \Omega,\\
\nu_{t=0}=(v^r_0, v^z_0).
\end{cases}
\end{equation}

\begin{proposition}\label{f2}
Given $\mathfrak{f}^2 \in L^{1}(0,T; H^4(\Omega))$ and suppose that $(v^r_0, v^z_0) \in H^4(\Omega)$. Then the problem \eqref{eqforv}
admits a unique solution $\nu=(v^r, v^z)$ that achieves
the initial data $(v^r_0, v^z_0)$ and  satisfies
\begin{equation}
\label{b2}
\norm{\nu(t)}_4  \leq \norm{(v^r_0, v^z_0)}_4 + \int_0^t\norm{\mathfrak{f}^2 }_4 \,d\tau .
\end{equation}
\end{proposition}

The last linear problem is the most substantial elliptic problem of $q$ with given $\mathfrak{f}^3$:
\begin{equation}
\label{eqforq}
\begin{cases}
-\dfrac{1}{\bar R}\bar \a_{\ell i}\partial_{a_i}(\bar R\bar \a_{\ell j}\partial_{a_j} q)= \mathfrak{f}^3 \,\, &\text{in } \Omega,\\
q=\mathfrak{f}^4 \,\, &\text{on } \Gamma.
\end{cases}
\end{equation}
\begin{proposition}\label{f3}
Given $\mathfrak{f}^3 \in H_{r,z}^3(\Omega), \mathfrak{f}^4\in H_z^{4.5}(\Gamma)$. Then the problem \eqref{eqforq}
admits a unique solution $q$ that satisfies
\begin{equation}\label{b3}
\norm{\nabla_{\bar\a} q}_{4}
 \leq  P\left(\norm{\bar\zeta}_{4}, \abs{b_0\cdot\nabla\bar\zeta}_{3.5}\right)(\norm{\mathfrak{f}^3}_{3}+\abs{\mathfrak{f}^4}_{4.5}).
\end{equation}
\end{proposition}
\subsubsection{Solvability of system \eqref{epsapproximate}}\label{solv es}
We employ a fixed point argument in order to produce a solution to the linear $\varepsilon$-approximate problem \eqref{epsapproximate}.

For $0 < T <1$ and $M>0$, we define the metric space in which to work:
\begin{equation}\label{metric_space_def}
\begin{split}
	\mathfrak{X}(M,T ) = &\left\{ (w, \pi, \xi)\;\vert\; w, \zeta, b_0\cdot\nabla\xi \in C([0,T]; H_{r,z}^4(\Omega)) \text{ satisfy that } (w, \xi)\mid_{t=0} = (v_0, \text{Id}) \right.
  \\  &\qquad\qquad\quad\text{and }\left.\norm{  (w, \xi, b_0\cdot\nabla\xi  )}_{L^\infty_TH_{r,z}^4}  +  \norm{\left( ( b_0\cdot\nabla)^2\xi , \nabla_{\bar \a}\pi \right) }_{L^{2}_TH_{r,z}^4} \le M .\right\}
  \end{split}
\end{equation}
Note that $\mathfrak{X}(M,T )$ is a Banach space. We then define a mapping $\mathcal{M}:\mathfrak{X}(M,T )\rightarrow\mathfrak{X}(M,T )$ as $\mathcal{M}(w, \pi, \xi)=(\nu, q, \zeta)$, where $\zeta, \nu$ and $ q$ are determined as follows. Given $(w, \pi, \xi)\in\mathfrak{X}(M,T )$, we first define $\zeta=(R, Z)$ as the solution to \eqref{eqforeta} with $\mathfrak{f}^1=w$ and the initial data $\eta_0=Id$, and $\nu=(v^r, v^z)$ as the solution to \eqref{eqforv} with $\mathfrak{f}^2=- \nabla_{\bar\a}\pi+(b_0\cdot \nabla)^2 \xi$ and the initial data $v_0$.
Next we define $q$ as the solution to \eqref{eqforq} with $$\mathfrak{f}^3=\partial_t\bar \a_{i\ell}\partial_{a_{\ell}} \nu_i- \bar\a_{i\ell}\partial_{a_\ell}\big((b_0\cdot\nabla)^2\zeta_i\big)+\dfrac{v^r\partial_t \bar R}{\bar R^2}-\pa^{\bar \a}_R\left(\dfrac{(\bar v^{\theta})^2}{\bar R}+\bar R(b_0\cdot\nabla\bar\Theta)^2\right)-\dfrac{(\bar v^{\theta})^2}{\bar R^2}+(b_0\cdot\nabla\bar\Theta)^2$$ and $\mathfrak{f}^4=\dfrac{1}{2}\dfrac{\bar C^2(t)}{\bar R^2}$, where $\nu$ and $\zeta$ are the functions constructed in the above.

Hence, if $M$ is taken to be sufficiently large with respect to $b_0,v_0, \bar R, \bar Z$ and $\varepsilon$ and then $0<T<1$ is taken to be sufficiently small (depending on $M$ and $\varepsilon$), then $(\nu, q, \zeta)\in \mathfrak{X}(M,T )$. This implies that the mapping $\mathcal{M}: \mathfrak{X}(M,T )\rightarrow \mathfrak{X}(M,T )$ is well-defined. And then we can show that the mapping $\mathcal{M}$ has a fixed
 point in the space $\mathfrak{X}(M,T )$ by proving the contraction and can verify that the unique fixed point $(\nu, q, \zeta)$ is a solution to \eqref{epsapproximate}. The details are omit here, the reader can refer to \cite[Section 5.1]{GW_16}.

Lastly, after we finding the solutions to the system \eqref{epsapproximate}, we derive an $\varepsilon$-independent estimates of the solutions, which allows us to pass to the limit as $\varepsilon \rightarrow 0$ to produce the solution to  the system \eqref{sub21}. Recalling that since $b_0^r=0$ on $\Gamma$, we do not need to impose boundary conditions for $(R, Z)$  and thus there is no boundary layer appearing  as $\varepsilon\rightarrow0$. The details are omit here. Finally, the existence of a unique solution to \eqref{sub21} is recorded in Theorem \ref{linearthm}.
\begin{theorem}\label{linearthm}
Suppose that the initial data $ (v^r_0, v^z_0) \in H_{r,z}^4(\Omega)$ with divergence free condition and that $(b_0^r, b_0^z) \in H_{r,z}^4(\Omega)$ satisfies \eqref{bcond}. Then there exists a $T(M)>0$ and a unique solution $(\nu, q, \zeta)$ to  \eqref{sub21} on $[0, T]$ that satisfy
\begin{equation*}
\norm{\nu(t)}_4^2+\norm{\zeta(t)}_4^2+\norm{b_0\cdot\nabla \zeta(t)}_4^2  \leq 2 M_0 .
\end{equation*}
\end{theorem}
\subsection{A priori estimates for approximate system \eqref{sub22}}\label{epes}
We derive the high order energy estimates for $(v^{\theta},\Theta)$ by standard energy method.
\begin{proposition}\label{vthe}
For $t\in [0,T]$, it holds that
\begin{equation}
\label{vtheq}
\norm{v^{\theta}(t)}_4^2+\norm{ Rb_0\cdot\nabla\Theta(t)}_4^2\leq M_0+TC(M)P\left(\sup_{t\in [0,T]}\mathfrak E(t)\right).
\end{equation}
\end{proposition}
\begin{proof}
Taking $L^2$ inner product with $v^{\theta}$ and using integration-by-parts yields
\begin{equation*}
	\begin{split}
	&\dfrac{1}{2}\dfrac{d}{dt}\int_{\Omega}|v^{\theta}|^2\,dx+\int_{\Omega} Rb_0\cdot\nabla\Theta b_0\cdot\nabla v^{\theta}\,dx\\&=\int_{\Omega}\left(-\dfrac{\bar v^{\theta}\bar v^r}{\bar R}
+b_0\cdot\nabla \bar Rb_0\cdot\nabla\bar\Theta\right) v^{\theta}\,dx
\\&=2\pi\int_{\mathbb T}\int_0^{R_0}\left(-\dfrac{\bar v^{\theta}\bar v^r}{\bar R}
+b_0\cdot\nabla \bar Rb_0\cdot\nabla\bar\Theta\right) v^{\theta}r\,drdz\\
&\leq C\norm{v^{\theta}}_0\norm{-\dfrac{\bar v^{\theta}\bar v^r}{\bar R}
+b_0\cdot\nabla \bar Rb_0\cdot\nabla\bar\Theta}_0
\leq C(\sqrt{M})\norm{v^{\theta}}_0
\end{split}
\end{equation*}
where we used Hardy's inequality
$$\norm{\dfrac{\bar v^{\theta}}{\bar R}}_0\ls \norm{\dfrac{\bar v^{\theta}}{r}}_0\norm{\dfrac{r}{\bar R}}_{L^{\infty}(\Omega)}\ls \norm{\bar v^{\theta}}_1, \norm{\dfrac{b_0\cdot\nabla \bar R}{\bar R}}_0\ls \norm{\dfrac{b_0\cdot\nabla \bar R}{r}}_0\norm{\dfrac{r}{\bar R}}_{L^{\infty}(\Omega)}\ls \norm{b_0\cdot\nabla \bar R}_1.$$
For $\int_{\Omega}\bar Rb_0\cdot\nabla\Theta b_0\cdot\nabla v^{\theta}\,dx$, we have
\begin{equation}
	\begin{split}
		\int_{\Omega} Rb_0\cdot\nabla\Theta b_0\cdot\nabla v^{\theta}\,dx=&\dfrac{1}{2}\dfrac{d}{dt}\int_{\Omega}\abs{ Rb_0\cdot\nabla\Theta}^2\,dx-\int_{\Omega} R b_0\cdot\nabla\Theta v^r b_0\cdot\nabla\Theta\,dx\\&+\int_{\Omega}b_0\cdot\nabla\Theta b_0\cdot\nabla R v^{\theta}\,dx\\\geq&\dfrac{1}{2}\dfrac{d}{dt}\int_{\Omega}\abs{Rb_0\cdot\nabla\Theta}^2\,dx- C(\sqrt{M})(\norm{ Rb_0\cdot\nabla\Theta}_0^2+\norm{v^{\theta}}_0^2).
\end{split}
	\label{fjiji}
\end{equation}
Hence, we have
\begin{equation*}
	\sup_{t\in [0,T]}\left(\norm{v^{\theta}}_0^2+\norm{Rb_0\cdot\nabla\Theta}_0^2\right)\leq M_0+TC(M)P\left(\sup_{t\in [0,T]}\mathfrak{E}(t)\right).
\end{equation*}
Acting $D^4$ on the second equation in \eqref{sub22}, taking $L^2$ innner product with $D^4 v^{\theta}$ and using integration-by-parts yields
\begin{equation*}
	\dfrac{1}{2}\dfrac{d}{dt}\int_{\Omega}|D^4 v^{\theta}|\,dx+\int_{\Omega}D^4(Rb_0\cdot\nabla\Theta)b_0\cdot\nabla D^4v^{\theta}\,dx=\int_{\Omega}\mathcal{G}D^4v^{\theta}\,dx,
\end{equation*}
where
\begin{equation*}
	\mathcal{G}:=[D^4, b_0\cdot\nabla]Rb_0\cdot\nabla\Theta+D^4\left(-\dfrac{\bar v^{\theta}\bar v^r}{\bar R}+b_0\cdot\nabla \bar Rb_0\cdot\nabla\bar\Theta\right).
\end{equation*}
By using the commutator estimate \eqref{co2}, we have
\begin{equation*}
	\begin{split}
\int_{\Omega}\mathcal{G}D^4 v^{\theta}\,dx&\leq \left(\norm{[D^4, b_0\cdot\nabla] Rb_0\cdot\nabla\Theta}_0+\norm{D^4\left(-\dfrac{\bar v^{\theta}\bar v^r}{\bar R}+b_0\cdot\nabla \bar Rb_0\cdot\nabla\bar\Theta\right)}_0\right)\norm{v^{\theta}}_4\\
&\leq C(\sqrt{M})P\left(\sup_{t\in [0,T]}\mathfrak{E}(t)\right),
	\end{split}
\end{equation*}
where we estimate
\begin{equation*}
	\begin{split}
	\norm{D^4(\dfrac{\bar v^{\theta}}{\bar R}\bar v^r)}_0&\leq C\left(\norm{\bar R D^4(\dfrac{\bar v^{\theta}}{\bar R})}_0\abs{\dfrac{\bar v^r}{\bar R}}_{L^{\infty}}+\norm{\bar v^r}_4\abs{\dfrac{\bar v^{\theta}}{\bar R}}_{L^{\infty}}\right)
	\\&\leq C\left(\norm{D^4\bar v^{\theta}-[D^4, \dfrac{\bar v^{\theta}}{\bar R}]\bar R}_0 \abs{\dfrac{\bar v^r}{\bar R}}_{L^{\infty}}+\norm{\bar v^r}_4\abs{\dfrac{\bar v^{\theta}}{\bar R}}_{L^{\infty}}\right) \leq C(M)
\end{split}
\end{equation*}
and
\begin{equation*}
	\begin{split}
		&\quad\norm{D^4(b_0\cdot\nabla\bar Rb_0\cdot\nabla\bar\Theta)}_0\\&\leq C\left(\norm{\bar R D^4(b_0\cdot\nabla\bar \Theta)}_0\abs{\dfrac{b_0\cdot\nabla \bar R}{\bar R}}_{L^{\infty}}+\norm{b_0\cdot\nabla\bar R}_4\abs{b_0\cdot\nabla\bar\Theta}_{L^{\infty}}\right)
	\\&\leq C\left(\norm{D^4(\bar Rb_0\cdot\nabla\bar\Theta)-[D^4, b_0\cdot\nabla\bar\Theta]\bar R}_0 \abs{\dfrac{b_0\cdot\nabla\bar R}{\bar R}}_{L^{\infty}}+\norm{b_0\cdot\nabla\bar R}_4\abs{b_0\cdot\nabla\bar\Theta}_{L^{\infty}}\right) \\&\leq C(M)
\end{split}
\end{equation*}
by using Hardy's inequality.

Using \eqref{rst} and \eqref{deflag}, we have
\begin{equation*}
	\begin{split}
		&\int_{\Omega}D^4(Rb_0\cdot\nabla\Theta)b_0\cdot\nabla D^4v^{\theta}\,dx\\=&\int_{\Omega}D^4( Rb_0\cdot\nabla\Theta)D^4\left(b_0\cdot\nabla(\bar R\partial_t\Theta)\right)\,dx\\&+\int_{\Omega}D^4( Rb_0\cdot\nabla\Theta)[D^4,b_0\cdot\nabla]v^{\theta}\,dx\\=&\dfrac{1}{2}\dfrac{d}{dt}\int_{\Omega}|D^4(R b_0\cdot\nabla\Theta)|^2\,dx-\int_{\Omega}D^4(Rb_0\cdot\nabla\Theta)D^4(v^r b_0\cdot\nabla\Theta)\,dx\\&+\int_{\Omega}D^4(Rb_0\cdot\nabla\Theta)\left(D^4(b_0\cdot\nabla R \dfrac{v^{\theta}}{ R})+[D^4,b_0\cdot\nabla]v^{\theta}\right)\,dx\\\geq& \dfrac{1}{2}\dfrac{d}{dt}\int_{\Omega}|D^4(Rb_0\cdot\nabla\Theta)|^2\,dx - C(\sqrt M)P\left(\sup_{t\in [0,T]}\mathfrak{E}(t)\right)
\end{split}
\end{equation*}
Hence, we arrive at the conclusion of this proposition.
\end{proof}
\subsection{Solvability of the system \eqref{sub22}}
The solvability of the system \eqref{sub22} can be established by a quite similar method we used in Section \ref{solv21}. We first construct strong solutions to the $\epsilon$-system:
\begin{equation*}
\begin{cases}
\partial_t \Theta-\epsilon b_0\cdot\nabla(Rb_0\cdot\nabla \Theta)=\dfrac{\bar v^{\theta}}{R}\\
\partial_t v^{\theta}-(b_0\cdot\nabla)(Rb_0\cdot\nabla\Theta)=-\dfrac{\bar v^{\theta}\bar v^r}{\bar R}
	+b_0\cdot\nabla \bar Rb_0\cdot\nabla\bar\Theta\,\,&\text{in}\,\,\Omega,\\(\Theta, v^{\theta})|_{t=0}=(\theta, v_0^{\theta}).
\end{cases}
\end{equation*}
and then by $\epsilon$-independent a priori estimates, we can pass to the limit as $\varepsilon \rightarrow 0$ to produce the solution to  the system \eqref{sub22}. The details are omit. Finally, the existence of a unique solution to \eqref{sub21} is recorded in Theorem \ref{linearthm2}.
\begin{theorem}\label{linearthm2}
Suppose that the initial data $v^{\theta}_0 \in H_{r,z}^4(\Omega)$ and that $(b_0^r, b_0^z) \in H_{r,z}^4(\Omega)$ satisfies \eqref{bcond}. Then there exists a $T(M)>0$ and a unique solution $(v^{\theta}, \Theta)$ to  \eqref{sub22} on $[0, T]$ that satisfy
\begin{equation*}
\norm{v^{\theta}(t)}_4^2+\norm{Rb_0\cdot\nabla \Theta(t)}_4^2 \leq 2 M_0 .
\end{equation*}
\end{theorem}
\subsection{Proof of Theorem \ref{th43}}
We now collect the estimates derived previously to conclude our estimates and also verify the a priori assumptions  \eqref{inin3}. That is, we shall now present the
\begin{proof}[Proof of Theorem \ref{th43}]
Combining the Proposition \ref{th412} and the Proposition \ref{vthe}, we get that
\begin{equation*}
\sup_{[0,T]}\mathfrak{E}(t)\leq M_0+C(M)TP\left(\sup_{t\in[0,T]} \mathfrak{E}(t)\right).
\end{equation*}
This provides us with a time of existence $T_1$ dependent of $M$ and an estimate on $[0,T_1]$ of the type:
\begin{equation}
\sup_{[0,T_1]}\mathfrak{E}(t)\leq 2M_0.
\end{equation}
Since by $\a^\kappa(0)=I$ and $J^\kappa(0)=1$, the bound \eqref{bound} and \eqref{pressure} verify in turn the a priori bounds \eqref{inin3} by the fundamental theorem of calculous with taking $T_1$ smaller if necessary.
The proof of Theorem \ref{th43} is thus completed.
\end{proof}
\section{Constructing solutions to the system \eqref{eq:mhd}}
\subsection{Iteration scheme}\label{diff}
In order to produce the solution to the nonlinear problem \eqref{eq:mhd}, we will pass to the limit as $n\rightarrow\infty$ in a sequence of approximate solutions $\{(v^r, v^z, v^{\theta}, R, Z, \Theta, q)^{(n)}\}_{n=0}^\infty$, which are constructed via the linearization iteration by using the linearized approximate problem \eqref{sub21}-\eqref{sub22}.
\subsection{}

In the following construction of $\{(v^r, v^z, v^{\theta}, R, Z, \Theta, q)^{(n)}\}_{n=0}^\infty$, we denote $\zeta^{(n)}=(R^{(n)}, Z^{(n)})$, $\nu^{(n)}=(v^r, v^z)^{(n)}$,  $\a^{(n)} =\a( \zeta^{(n)})$. We set $(\nu, v^{\theta}, \zeta, \Theta, q)^{(0)} =(\nu, v^{\theta}, \zeta, \Theta, q)^{(1)} =(0, 0, 0, r, z, \theta, 0)$, and hence $\a^{(0)}=\a^{(1)}=I$. Note that $\dt (\zeta, \Theta)^{(1)}=\left(\nu, \frac{v^{\theta}}{R}\right)^{(1)}$. Now suppose that $(\nu, v^{\theta}, \zeta, \Theta, q)^{(n)}$, $n\geq 1$ are given such that $(\nu, v^{\theta}, b_0\cdot\nabla \zeta, Rb_0\cdot\nabla\Theta)^{(n)}$, $\partial_t (\zeta, \Theta)^{(n)}=(\nu, \frac{v^{\theta}}{R})^{(n)} \in L^{\infty}(0,T;H^4_{r,z}(\Omega))$, $  (\zeta, \Theta)^{(n)}\mid_{t=0}=Id$ and that
\begin{align}
\abs{ J^{(n)}(t)-1}\leq \dfrac{1}{8} \text{ and } \abs{\a_{ij}^{(n)}(t)-\delta_{ij}}\leq \dfrac{1}{8} \text{ in }\Omega.
\end{align}
Then by Theorem \ref{linearthm} and Theorem \ref{linearthm2}, we can construct $(v^r, v^z, v^{\theta}, R, Z, \Theta, q)^{(n+1)}$ as the solution to \eqref{sub21} and \eqref{sub22} with $(\bar v^r, \bar v^z, \bar v^{\theta}, \bar R, \bar Z, \bar \Theta)=(\nu, v^{\theta}, \zeta, \Theta)^{(n)}$. That is,
\begin{equation}
	\begin{cases}
		\partial_t \zeta^{(n+1)}= \nu^{(n+1)}\,\,&\text{in}\,\,\Omega,\\
\partial_t \nu^{(n+1)}+\nabla_{\a^{(n)}} q^{(n+1)}-(b_0\cdot\nabla)^2\zeta^{(n+1)}
=\left(\dfrac{(v^{\theta(n)})^2}{ R^{(n)}}- R^{(n)}(b_0\cdot\nabla \Theta^{(n)})^2,0\right)\,\,&\text{in}\,\,\Omega,\\
\Div_{\a^{(n)}}\nu^{(n+1)}=0\,\,&\text{in}\,\,\Omega,\\
q^{(n+1)}=\dfrac{1}{2}\dfrac{ C^{(n)2}(t)}{R^{(n)2}}\,\,&\text{on}\,\,\Gamma,\\
(\zeta^{(n+1)}, \nu^{(n+1)})|_{t=0}=(r, z, v_0^r, v_0^z).\,\,&\text{in}\,\,\Omega.
\end{cases}
	\label{its1}
\end{equation}
and
\begin{equation}
	\begin{cases}
		\partial_t \Theta^{(n+1)}=\dfrac{v^{\theta(n+1)}}{R^{(n+1)}}\,\,&\text{in}\,\,\Omega,\\
	\partial_t v^{\theta(n+1)}-(b_0\cdot\nabla)(R^{(n+1)}b_0\cdot\nabla\Theta^{(n+1)})=-\dfrac{v^{\theta(n)}v^{r(n)}}{R^{(n)}}
	+b_0\cdot\nabla R^{(n)}b_0\cdot\nabla\Theta^{(n)}\,\,&\text{in}\,\,\Omega,\\(\Theta^{(n+1)}, v^{\theta(n+1)})|_{t=0}=(\theta, v_0^{\theta}).
\end{cases}
	\label{its2}
\end{equation}

We define a higher order energy functional
\begin{equation}
\mathfrak{E}^{(n)}(t):=\norm{(v^r, v^z, v^{\theta})^{(n)}(t)}_4^2+\norm{(b_0\cdot\nabla R, b_0\cdot\nabla Z, Rb_0\cdot\nabla \Theta)^{(n)}(t)}_4^2.
\end{equation}
We claim that by taking $T$ sufficiently small (only depends on $M_0$), it holds that
\begin{equation}\label{claim1}
\sup_{t\in[0,T]}\mathfrak{E}^{(n)}(t)\leq 2M_0.
\end{equation}
We shall prove the claim \eqref{claim1} by the induction. First it holds for $n=0,1$. Now suppose that it holds for $ n\le m$ for $m\ge 1$, then we prove that it holds for $n=m+1$. First, it holds from the induction assumption that
\begin{equation}
\norm{\left(\partial_t R^{(m)}, \partial_t Z^{(m)}\right)}_{L^\infty_TH_{r,z}^4}\le P\left(M_0 \right).
\end{equation}
Then from the estimates \eqref{nueq} and \eqref{vtheq}, we obtain
\begin{equation}
\sup_{t\in[0,T]}\mathfrak{E}^{(m+1)}(t)\leq M_0+\sqrt TP\left(\sup_{t\in[0,T]}\mathfrak{E}^{(m+1)}(t)\right)P\left(M_0 \right).
\end{equation}
Taking $T$ sufficiently small (only depends on $M_0$), we conclude the claim \eqref{claim1}.

Now we will prove that the sequence $\{(v^r, v^z, v^{\theta}, R, Z, \Theta, q)^{(n)}\}_{n=0}^\infty$ converges in certain strong norm. Let $n\ge 3$ and denote the differences:
\begin{equation}
\begin{split}
&\tilde \nu^{(n)}=\nu^{(n+1)}-\nu^{(n)}, \quad\tilde v^{\theta}=v^{\theta(n+1)}-v^{\theta(n)},\quad\tilde q^{(n)}=q^{(n+1)}-q^{(n)}, \\&\tilde \zeta^{(n)}=\zeta^{(n+1)}-\zeta^{(n)}, \quad\tilde \Theta^{(n)} =\Theta^{(n+1)}-\Theta^{(n)}.
\end{split}
\end{equation}
Also, we denote
\begin{equation}
\tilde\a^{(n)}=\a^{(n)}-\a^{(n-1)},\quad \tilde C^{(n)}=C^{(n)}-C^{(n-1)}.
\end{equation}
We find that
\begin{equation}
\label{diffe}
\begin{cases}
\partial_t\tilde \zeta^{(n)}=\tilde \nu^{(n)} &\text{in } \Omega\\
\partial_t\tilde \nu^{(n)}_i+\a^{(n)}_{ij}\partial_{a_j}\tilde q^{(n)} =(b_0\cdot \nabla)^2 \tilde \zeta^{(n)}_i-\tilde\a^{(n)}_{ij}\partial_{a_j} q^{(n)}+\tilde I_i^{(n-1)} &\text{in } \Omega,\\
\Div_{\a^{(n)}}\tilde \nu^{(n)}=-\tilde\a^{(n)}_{ij}\partial_{a_j} \nu^{(n)}_i+\dfrac{\tilde R^{(n-1)}v^{r(n)}}{R^{(n)}R^{(n-1)}} &\text{in } \Omega,\\
\tilde q^{(n)}=\dfrac{1}{2}\left(\dfrac{\tilde C^{(n-1)}}{R^{(n)}}-\dfrac{\tilde R^{(n-1)}C^{(n-1)}}{R^{(n)}R^{(n-1)}}\right)\left(\dfrac{C^{(n)}}{R^{(n)}}+\dfrac{C^{(n-1)}}{R^{(n-1)}}\right) &\text{on }\Gamma,\\
 (\tilde \nu, \tilde \zeta)^{(n)}|_{t=0}=0.
\end{cases}
\end{equation}
and
\begin{equation}
\label{diffe2}
	\begin{cases}
		\partial_t \tilde\Theta^{(n)}=\dfrac{\tilde v^{\theta(n)}R^{(n)}-\tilde R^{(n)}v^{\theta(n)}}{R^{(n)}R^{(n+1)}}\,\,&\text{in}\,\,\Omega,\\
	\partial_t \tilde v^{\theta(n)}-(b_0\cdot\nabla)\left(R^{(n+1)}b_0\cdot\nabla\tilde\Theta^{(n)}-\tilde R^{(n)}b_0\cdot\nabla\Theta^{(n)}\right)=\tilde I_{3}^{(n-1)}\,\,&\text{in}\,\,\Omega,\\(\tilde\Theta^{(n)}, \tilde v^{\theta(n)})|_{t=0}=0,
\end{cases}
\end{equation}
where
\begin{equation}
\begin{split}
\tilde I_1^{(n-1)}=&-\dfrac{\tilde R^{(n-1)}(v^{\theta(n)})^2}{R^{(n)}R^{(n-1)}}+\dfrac{\tilde v^{\theta(n-1)}(v^{\theta(n)}+v^{\theta(n-1)})}{R^{(n-1)}}\\&-\tilde R^{(n-1)}(b_0\cdot\nabla\Theta^{(n)})^2-b_0\cdot\nabla \tilde\Theta^{(n-1)}R^{(n)}\dfrac{R^{(n-1)}}{R^{(n)}}(b_0\cdot\nabla\Theta^{(n-1)}+b_0\cdot\nabla\Theta^{(n)}),\\
\tilde I_2^{(n-1)}=& 0,\\
\tilde I_3^{(n-1)}=&-\dfrac{\tilde v^{\theta(n-1)}v^{r(n)}+v^{\theta(n-1)}\tilde v^{r(n-1)}}{R^{(n)}}+\dfrac{\tilde R^{(n-1)}v^{\theta(n-1)}v^{r(n-1)}}{R^{(n)}R^{(n-1)}}\\
	&+b_0\cdot\nabla \tilde R^{(n-1)}b_0\cdot\nabla\Theta^{(n)}+b_0\cdot\nabla R^{(n-1)}b_0\cdot\nabla\tilde \Theta^{(n-1)}.
\end{split}
\end{equation}

We will now estimate the differences in $H^3_{r,z}$ norm. We will use a similar strategy which was used in Sections \ref{apest} and \ref{epes}, and we divide our estimates into several steps.

{\it Step 1: preliminary estimates of $\tilde\a^{(n)}$ and $\tilde C^{(n)}$.}  First, note that
\begin{equation}
\begin{split}
\quad\tilde\a^{(n)}_{ij}(t) &= \int_0^t\partial_t\left (\a^{(n)}_{ij}-\a^{(n-1)}_{ij}\right )\,d\tau
 \\&=\int_0^t \bigg(\tilde\a^{(n)}_{i\ell}\left(\partial_t\partial_{a_\ell}\zeta^{(n)}_i\a^{(n)}_{ij}\right)
+\a^{(n-1)}_{i\ell}\left(\partial_t\partial_{a_\ell}\tilde \zeta_i^{(n-1)}\a^{(n)}_{ij}\right)\\
&\quad+\a^{(n-1)}_{i\ell}\left(\partial_t\partial_{a_\ell}\zeta_i^{(n-1)}\tilde\a^{(n)}_{ij}\right) \bigg)\,d\tau.
\end{split}
\end{equation}
Hence, we have, by the first equation of \eqref{diffe},
\begin{equation}\label{dest1}
\begin{split}
\norm{\tilde\a^{(n)}_{ij}(t)}_2^2
&\le P(M_0)T^2\left( \norm{\tilde\a^{(n)}_{ij}}_{L^{\infty}_TH_{r,z}^2}^2+\norm{\partial_t \tilde \zeta^{(n-1)}}_{L^{\infty}_TH_{r,z}^3}^2\right)
\\&\le P(M_0)T^2\left(\norm{\tilde\a^{(n)}_{ij}}_{L^{\infty}_TH_{r,z}^2}^2+\norm{\tilde \nu^{(n-1)}}_{L^{\infty}_TH_{r,z}^3}^2\right).
\end{split}
\end{equation}
Next, we estimate $\tilde C^{(n)}(t)$. Recalling the definition of $C^{(n)}(t), C^{(n-1)}(t)$ and using Taylor's formula, we have
\begin{equation}
\begin{split}
\tilde C^{(n)}(t)&=C(0)(e^{\int_0^tA^{(n)}(\tau)\,d\tau}-e^{\int_0^tA^{(n-1)}(\tau)\, d\tau})\\&=C(0)e^{\alpha(t)}\int_0^t\left(A^{(n)}- A^{(n-1)}\right)\,d\tau\\&\leq C(0)e^{\alpha(t)}\int_0^tC(M_0)\norm{\tilde \nu^{(n-1)},\tilde R^{(n-1)}}_2^2
\end{split}
\end{equation}
where $$\sup_{t\in[0,T]}\abs{\alpha(t)}\leq \sup_{t\in[0,T]}\abs{A^{(n)}(t)}+\sup_{t\in[0,T]}\abs{A^{(n-1)}(t)}\leq C(M_0).$$
Consequently, we have
\begin{equation}
\label{cest}
\sup_{t\in [0,T]}\abs{\tilde C^{(n)}(t)}^2\leq P(M_0)T^2\norm{\tilde \nu^{(n-1)}, \tilde R^{(n-1)}}_2^2.
\end{equation}

{\it Step 2: estimates of $\tilde \zeta^{(n)}$.}  By the first equation of \eqref{diffe}, we have
\begin{equation}
\label{dest2}
\norm{\tilde \zeta^{(n)}(t)}_3^2\leq T^2\norm{\tilde \nu^{(n)}}_{L^{\infty}_TH_{r,z}^3}^2.
\end{equation}

{\it Step 3: estimates of $\tilde q^{(n)}$.} Applying $J^{(n)}\Div_{\a^{(n)}}$ to the second equation of \eqref{diffe}, we have
\begin{equation}
\begin{split}
&\dfrac{1}{R^{(n)}}\partial_{a_\ell}(R^{(n)}J^{(n)}\a^{(n)}_{i\ell}\a^{(n)}_{ij}\partial_{a_j} \tilde q^{(n)})\\=&-J^{(n)}\partial_t\a^{(n)}_{i\ell}\partial_{a_\ell} \tilde \nu^{(n)}_i+\dfrac{J^{(n)}v^{r(n)}\tilde v^{r(n)}}{R^{(n)2}}-J^{(n)}\partial_t\left(\tilde\a^{(n)}_{i\ell}\partial_{a_\ell} \nu^{(n)}_i+\dfrac{\tilde R^{(n-1)}v^{r(n)}}{R^{(n)}R^{(n-1)}}\right)
\\&-J^{(n)}\a^{(n)}_{i\ell}\partial_{a_\ell}\left(\tilde\a^{(n)}_{ij}\partial_{a_j}q^{(n)}\right)-\dfrac{J^{(n)}}{R^{(n)}}\tilde\a^{(n)}_{1j}\partial_{a_j}q^{(n)}
+\dfrac{J^{(n)}}{R^{(n)}}\left((b_0\cdot\nabla)^2\tilde R^{(n)}\right)\\&+\a^{(n)}_{i\ell}\partial_{a_\ell}\left((b_0\cdot\nabla)^2\tilde \zeta^{(n)}\right)+\left(\dfrac{J^{(n)}}{R^{(n)}}+J^{(n)}\a^{(n)}_{1\ell}\partial_{a_\ell}\right)\tilde I_1^{(n-1)}.
\end{split}
\end{equation}
By the similar arguments in Section \ref{pressure1}, we obtain
\begin{equation}
\label{dpr}
\begin{split}
\norm{\tilde q^{(n)}}_3^2\le P(M_0) \bigg( &\norm{\left( \tilde \nu^{(n)},b_0\cdot\nabla\tilde \zeta^{(n)}, \tilde R^{(n-1)}, \tilde v^{\theta(n-1)}, R^{(n)}b_0\cdot\nabla\tilde\Theta^{(n-1)}\right) }_{3}^2+\norm{\tilde\a^{(n)}}_{2}^2\bigg).
\end{split}
\end{equation}

{\it Step 4: tangential energy estimates.} We shall again use Alinac's good unknowns:
\begin{equation*}
\begin{split}
\mathcal V^{(n+1)}=\bp^3\nu^{(n+1)}-\bp^3\zeta_i^{(n)}\a^{(n)}_{ij}\partial_{a_j} \nu^{(n+1)},\quad
\mathcal Q^{(n+1)}=\bp^3q^{(n+1)}-\bp^3\zeta_i^{(n)}\a^{(n)}_{ij}\partial_{a_j} q^{(n+1)}.
\end{split}
\end{equation*}
We denote $\tilde{\mathcal V}^{(n)}=\mathcal V^{(n+1)}-\mathcal V^{(n)}, \tilde{\mathcal Q}^{(n)}=\mathcal Q^{(n+1)}-\mathcal Q^{(n)}$, then we have that
\begin{equation}
\label{p1}
\partial_t\tilde{\mathcal V}^{(n)}_i+\a^{(n)}_{ij}\partial_{a_j} \tilde{\mathcal Q}^{(n)}-(b_0\cdot\nabla)\bp^3(b_0\cdot\nabla\tilde \zeta_i^{(n)})=-\tilde\a^{(n)}_{ij}\partial_{a_ j}\mathcal Q^{(n)}+ f_i^{(n)}
\end{equation}
and
\begin{equation*}
\Div_{\a^{(n)}} \tilde{\mathcal V}=-\tilde\a^{(n)}_{ij}\partial_{a_j} \mathcal V_i^{(n)}+g^{(n)},
\end{equation*}
where
\begin{align*}
f_i^{(n)}=&\partial_t\left(\bp^3{\tilde\zeta^{(n-1)}}_m\a^{(n)}_{mj}\partial_{a_j}\nu^{(n+1)}_i
+\bp^3\zeta^{(n-1)}_m\tilde\a^{(n)}_{mj}\partial_{a_j}\nu^{(n+1)}_i+\bp^3\zeta^{(n-1)}_m\a^{(n-1)}_{mj}\partial_{a_j}\tilde \nu^{(n)}_i\right)
\\&+\tilde\a^{(n)}_{mj}\partial_{a_j}\left(\a^{(n)}_{i\ell}\partial_{a_\ell} q^{(n+1)}\right)\bp^3\zeta^{(n)}_m
+\a^{(n-1)}_{mj}\partial_{a_j}\left(\tilde\a^{(n)}_{i\ell}\partial_{a_\ell} q^{(n+1)}\right)\bp^3\zeta^{(n)}_m
\\&+\a^{(n-1)}_{mj}\partial_{a_j}\left(\a^{(n-1)}_{i\ell}\partial_{a_\ell}\tilde q^{(n)}\right)\bp^3\zeta^{(n)}_m+\a^{(n-1)}_{mj}\partial_{a_j}\left(\a^{(n-1)}_{i\ell}\partial_{a_\ell} q^{(n)}\right)\bp^3\tilde\zeta^{(n-1)}_m
\\&-\left[\bp^{2},\tilde\a^{(n)}_{mj}\a^{(n)}_{i\ell}\bp\right]\partial_{a_\ell}\zeta^{(n)}_m\partial_{a_j} q^{(n+1)}
-\left[\bp^{2},\a^{(n-1)}_{mj}\tilde\a^{(n)}_{i\ell}\bp\right]\partial_{a_\ell}\zeta^{(n)}_m\partial_{a_j} q^{(n+1)}
\\
&-\left[\bp^{2},\a^{(n-1)}_{mj}\a^{(n-1)}_{i\ell}\bp\right]\partial_{a_\ell}\tilde\eta^{(n-1)}_m\partial_{a_j} q^{(n+1)}-\left[\bp^{2},\a^{(n-1)}_{mj}\a^{(n-1)}_{i\ell}\bp\right]\partial_{a_\ell}\zeta^{(n-1)}_m\partial_{a_j} \tilde q^{(n)}
\\
&-\left[\bp^3, \tilde\a^{(n)}_{ij},\partial_{a_j}\right] q^{(n+1)}-\left[\bp^3, \a^{(n-1)}_{ij},\partial_{a_j}\right] \tilde q^{(n)}+\left[\bp^3, b_0\cdot\nabla\right]b_0\cdot\nabla\tilde\zeta^{(n)}_i+\bp^3\tilde I_i^{(n-1)}
\end{align*}
and
\begin{align*}
g^{(n)}=&-\left[\bp^2, \tilde\a^{(n)}_{mj}\a^{(n)}_{i\ell}\bp\right]\partial_{a_\ell}\zeta^{(n)}_m\partial_{a_j} \nu^{(n+1)}_i-\left[\bp^2, \a^{(n-1)}_{mj}\tilde\a^{(n)}_{i\ell}\bp\right]\partial_{a_\ell}\zeta^{(n)}_m\partial_{a_j}\nu^{(n+1)}_i
\\&-\left[\bp^2, \a^{(n-1)}_{mj}\a^{(n-1)}_{i\ell}\bp\right]\partial_{a_\ell}\tilde\zeta^{(n-1)}_m\partial_{a_j}\nu^{(n+1)}_i-\left[\bp^2, \a^{(n-1)}_{mj}\a^{(n-1)}_{i\ell}\bp\right]\partial_{a_\ell}\zeta^{(n-1)}_m\partial_{a_j} \tilde\nu^{(n)}_i
\\&-\left[\bp^3, \tilde\a^{(n)}_{ij},\partial_{a_j}\right]\nu^{(n+1)}_i-\left[\bp^3, \a^{(n-1)}_{ij},\partial_{a_j}\right]\tilde \nu^{(n)}_i+\tilde\a^{(n)}_{mj}\partial_{a_j}\left(\a^{(n)}_{i\ell}\partial_{a_\ell}\nu^{(n+1)}_i\right)\bp^3\zeta^{(n)}_m
\\
&+\a^{(n-1)}_{mj}\partial_{a_j}\left(\tilde\a^{(n)}_{i\ell}\partial_{a_\ell}\nu^{(n+1)}_i\right)\bp^3\zeta^{(n)}_m+\a^{(n-1)}_{mj}\partial_{a_j}\left(\a^{(n-1)}_{i\ell}\partial_{a_\ell}\tilde\nu^{(n)}_i\right)\bp^3\zeta^{(n)}_m\\&+\a^{(n-1)}_{mj}\partial_{a_j}\left(\a^{(n-1)}_{i\ell}\partial_{a_\ell} \nu^{(n)}_i\right)\bp^3\tilde\zeta^{(n-1)}_m+\bp^3\dfrac{\tilde R^{(n-1)}v^{r(n)}}{R^{(n)}R^{(n-1)}}.
\end{align*}
Then we find
\begin{equation}\label{ttttt1}
\begin{split}
 &\dfrac{1}{2}\dfrac{d}{dt}\int_{\Omega}\abs{\tilde{\mathcal V}^{(n)}}^2+\int_{\Omega} \nabla_{\a^{(n)}} \tilde{\mathcal Q}^{(n)}\cdot \tilde{\mathcal V}^{(n)} +\int_{\Omega} \bp^3\left(b_0\cdot\nabla\tilde\zeta^{(n)}_i\right)   b_0\cdot\nabla \tilde{\mathcal V}^{(n)}_i
 \\&\quad=\int_{\Omega} \left(f_i^{(n)}-\tilde\a^{(n)}_{ij}\partial_{a_j}\mathcal Q^{(n)}\right)\tilde{\mathcal V}^{(n)}_i.
 \end{split}
\end{equation}

We estimate in an elementary way as usual to deduce that
 \begin{align}
 \label{hg}
&\int_{\Omega} \left(f_i^{(n)}-\tilde\a^{(n)}_{ij}\partial_{a_j}\mathcal Q^{(n)}\right)\tilde{\mathcal V}^{(n)}_i
\\& \quad\le P(M_0) \bigg(
 \norm{\left(\tilde \nu^{(n)}, b_0\cdot\nabla\tilde\zeta^{(n)}, \tilde \nu^{(n-1)}, \tilde\zeta^{(n-1)}, \tilde v^{\theta(n-1)}, R^{(n)}b_0\cdot\nabla\tilde\Theta^{(n-1)}\right)}_{3}^2+\norm{\tilde\a^{(n)}}_{2}^2\bigg).
\nonumber
 \end{align}
By the integration by parts, recalling the boundary condition for $\tilde q^{(n)}, q^{(n)}$ and using the estimate \eqref{cest} and \eqref{dpr}, we obtain
\begin{align}
\nonumber
& \int_{\Omega} \nabla_{\a^{(n)}} \tilde{\mathcal Q}^{(n)}\cdot \tilde{\mathcal V}^{(n)}\\\nonumber&\quad =2\pi\int_{\mathbb T}R_0\tilde{\mathcal Q}^{(n)}\a^{(n)}_{i1}\tilde{\mathcal V}^{(n)}_i \,dz-\int_{\Omega} \tilde{\mathcal Q}^{(n)} \Div_{\a^{(n)}}\tilde{\mathcal V}^{(n)}-\int_{\Omega}\pa_{a_\ell} \left(\dfrac{r}{R^{(n)}}\a^{(n)}_{i\ell} \right)\tilde{\mathcal Q}^{(n)}\dfrac{R^{(n)}}{r}\tilde{\mathcal V}^{(n)}_i \\\nonumber&\quad=
- \int_{\mathbb T} R_0\left(\bp^3\tilde q^{(n)}-\bp^3\zeta_m^{(n-1)}\a_{mj}^{(n)}\partial_{a_j}\tilde q^{(n)}\right)\a^{(n)}_{i1} \tilde{\mathcal V}^{(n)}_i\,dz\\\nonumber&\qquad+\int_{\mathbb T} R_0\left(\bp^3\tilde\zeta_m^{(n-1)}\a_{mj}^{(n)}\partial_{a_j}q^{(n+1)}+\bp^3\zeta_m^{(n-1)}\tilde\a_{mj}^{(n)}\partial_{a_j}q^{(n+1)}\right) \a^{(n)}_{i1} \tilde{\mathcal V}^{(n)}_i \,dz\\\nonumber&\qquad+\int_{\Omega} \left(\tilde{\mathcal{Q}}^{(n)}\tilde\a^{(n)}_{i\ell}\partial_{a_\ell}\mathcal V^{(n)}_i-\tilde{\mathcal Q}^{(n)}g^{(n)}-\pa_{a_\ell} \left(\dfrac{r}{R^{(n)}}\a^{(n)}_{i\ell} \right)\tilde{\mathcal Q}^{(n)} \dfrac{R^{(n)}}{r}\tilde{\mathcal V}^{(n)}_i\right)
\\&\nonumber\quad\le P(M_0) \left(\abs{\bp^2(b_0\cdot\nabla\tilde\zeta)^{(n-1)}+\tilde C^{(n-1)}}_{1/2}\abs{\tilde{\mathcal V}^{(n)}}_{-1/2}+\norm{\tilde\a^{(n)}}_2\abs{\tilde{\mathcal V}^{(n)}}_{-1/2}\right.
\\&\nonumber\qquad\qquad\qquad\left.+\norm{\left(\tilde \nu^{(n)}, b_0\cdot\nabla\tilde\zeta^{(n)},\tilde\zeta^{(n-1)},\tilde v^{\theta(n-1)}, R^{(n)}b_0\cdot\nabla\tilde\Theta^{(n-1)}\right)}_{3}^2+\norm{\tilde\a^{(n)}}_{2}^2\right)
\\&\quad\le P(M_0)\left(\norm{\left(\tilde \nu^{(n)}, b_0\cdot\nabla\tilde\zeta^{(n)},\tilde\nu^{(n-1)},\tilde\zeta^{(n-1)},b_0\cdot\nabla\tilde\zeta^{(n-1)},\tilde v^{\theta(n-1)}, R^{(n)}b_0\cdot\nabla\tilde\Theta^{(n-1)}\right)}_{3}^2\right.
\\&\nonumber\qquad\qquad\qquad\left.+\norm{\tilde\a^{(n)}}_{2}^2\right).
\label{hg2}
\end{align}
By using the first equation of \eqref{diffe}, we have
\begin{equation}
\begin{split}
 &\int_{\Omega}  \bp^3(b_0\cdot\nabla\tilde\zeta^{(n)}_i)    b_0\cdot\nabla \tilde{\mathcal V}^{(n)}_i
\\ &\quad=\hal \dfrac{d}{dt}\int_{\Omega}\abs{\bp^3(b_0\cdot\nabla\tilde\eta^{(n)})}^2
\\&\qquad+\int_{\Omega} \bp^3(b_0\cdot\nabla\tilde\zeta^{(n)}_i) \left([\bp^3,b_0\cdot\nabla] \tilde \nu^{(n)}_i+b_0\cdot\nabla\left(\bp^3\tilde\zeta^{(n-1)}_m\a^{(n)}_{mj}\partial_{a_j}\nu^{(n+1)}_i\right)\right)
\\&\qquad+\int_{\Omega} \bp^3(b_0\cdot\nabla\tilde\zeta^{(n)}_i) b_0\cdot\nabla\left(
\bp^3\zeta^{(n-1)}_m\tilde\a^{(n)}_{mj}\partial_{a_j}\nu^{(n+1)}_i+\bp^3\zeta^{(n-1)}_m\a^{(n-1)}_{mj}\partial_{a_j} \tilde \nu^{(n)}_i\right)
\\ &\quad\ge\hal \dfrac{d}{dt}\int_{\Omega}\abs{\bp^3(b_0\cdot\nabla\tilde\zeta^{(n)})}^2
- P(M_0)\bigg(\norm{\left(\tilde \nu^{(n)},b_0\cdot\nabla\tilde\zeta^{(n)}, b_0\cdot\nabla\tilde\zeta^{(n-1)}\right)}_3^2
+\norm{\tilde\a^{(n)}}_2^2\bigg).
\end{split}
\label{hg3}
\end{equation}
Hence, integrating \eqref{ttttt1} in time directly  and then using the estimates \eqref{hg}--\eqref{hg3}, we obtain
\begin{align*}
\nonumber&\norm{\tilde{\mathcal V}^{(n)}(t)}_0^2+\norm{\bp^3 (b_0\cdot\nabla \tilde\zeta^{(n)})(t)}_0^2
\\&\leq P(M_0) T^2\left(\norm{\left(\tilde \nu^{(n)}, b_0\cdot\nabla\tilde\zeta^{(n)},\tilde\nu^{(n-1)},\tilde\zeta^{(n-1)},b_0\cdot\nabla\tilde\zeta^{(n-1)},\tilde v^{\theta(n-1)}, R^{(n)}b_0\cdot\nabla\tilde\Theta^{(n-1)}\right)}_{L^{\infty}_TH_{r,z}^3}^2\right.
\\&\nonumber\qquad\qquad\qquad\left.+\norm{\tilde\a^{(n)}}_{L^{\infty}_TH_{r,z}^2}^2\right).
\end{align*}
By the definition of $\tilde{\mathcal{V}}^{(n)}$, we have
\begin{align}\label{dest3}
\nonumber&\norm{\bp^3\tilde \nu^{(n)}(t)}_0^2+\norm{\bp^3(b_0\cdot\nabla\tilde\zeta^{(n)})(t)}_0^2\\&\leq P(M_0) T^2\left( \norm{\left(\tilde \nu^{(n)}, b_0\cdot\nabla\tilde\zeta^{(n)},\tilde\nu^{(n-1)},\tilde\zeta^{(n-1)},b_0\cdot\nabla\tilde\zeta^{(n-1)},\tilde v^{\theta(n-1)}, R^{(n)}b_0\cdot\nabla\tilde\Theta^{(n-1)}\right)}_{L^{\infty}_TH^3_{r,z}}^2\right.
\\&\qquad\qquad\qquad\left.+\norm{\tilde\a^{(n)}}_{L^{\infty}_TH^2_{r,z}}^2\right).\nonumber
\end{align}

{\it Step 5: curl and divergence estimates.} It follows from \eqref{diffe} that
\begin{align*}
&\Div_{\a^{(n)}}\tilde \nu^{(n)}=-\tilde\a^{(n)}_{ij}\partial_{a_j} \nu^{(n)}_i+\dfrac{\tilde R^{(n-1)}v^{r(n)}}{R^{(n)}R^{(n-1)}},\\
&\partial_t\left(\Div_{\a^{(n)}}b_0\cdot\nabla\tilde\zeta^{(n)}\right)=h^{(n)},
\\&
 \partial_t\left(\curl_{\a^{(n)}}\tilde \nu^{(n)}\right)-b_0\cdot\nabla\left(\curl_{\a^{(n)}}\left(b_0\cdot\nabla \tilde\zeta^{(n)}\right)\right) =\phi^{(n)},
\end{align*}
where
\begin{equation*}\nonumber
\begin{split}
h^{(n)}=&\partial_t \a^{(n)}_{ij}\partial_{a_j} \left(b_0\cdot\nabla\tilde\zeta^{(n)}_i\right)+\left[\Div_{\a^{(n)}}, b_0\cdot\nabla\right] \tilde\nu^{(n)}-b_0\cdot\nabla\left(\tilde\a^{(n)}_{ij}\partial_{a_j} \nu^{(n)}_i-\dfrac{\tilde R^{(n-1)}v^{r(n)}}{R^{(n)}R^{(n-1)}}\right),
\\
\phi^{(n)}=&\left(\left[\epsilon_{3j\ell}\a^{(n)}_{jm}\partial_{a_m}, b_0\cdot\nabla\right](b_0\cdot\nabla \tilde\zeta^{(n)}_{\ell})\right)+\epsilon_{3j\ell} \partial_t\a^{(n)}_{jm} \partial_{a_m}\tilde \nu^{(n)}_\ell-\epsilon_{3j\ell} \tilde\a^{(n)}_{jm} \partial_t\partial_{a_m}\tilde \nu^{(n)}_\ell\\&+b_0\cdot\nabla\left(\epsilon_{3j\ell}\tilde\a^{(n)}_{jm} \partial_{a_m}\left(b_0\cdot\nabla\zeta^{(n)}_\ell\right)\right)+\epsilon_{3j\ell}\left[\tilde\a^{(n)}_{jm}\partial_{a_m}, b_0\cdot\nabla\right]b_0\cdot\nabla\zeta^{(n)}_\ell\\&+\epsilon_{3j\ell}\a^{(n)}_{jm}\partial_{a_m}\tilde I_{\ell}^{(n-1)}+\epsilon_{3j\ell}\tilde\a^{(n)}_{jm}\partial_{a_m}\left(\dfrac{(v^{\theta(n-1)})^2}{R^{(n-1)}}-R^{(n-1)}(b_0\cdot\nabla\Theta^{(n-1)})^2\right)
\end{split}
\end{equation*}
where $\epsilon_{3j\ell}=0$ for $j=\ell$ and $\epsilon_{312}=1, \epsilon_{321}=-1$.
Then following a similar argument in Section \ref{curle}, we can obtain
\begin{align}
\label{dest4}
&\norm{\left(\Div \tilde\nu^{(n)} ,\curl \tilde \nu^{(n)},\Div (b_0\cdot\nabla\tilde\zeta^{(n)}) ,\curl(b_0\cdot\nabla\tilde\zeta^{(n)})\right)(t)}_2^2\nonumber\\&\leq P(M_0) T^2\left( \norm{\left(\tilde \nu^{(n)}, b_0\cdot\nabla\tilde\zeta^{(n)},\tilde\nu^{(n-1)},\tilde\zeta^{(n-1)},b_0\cdot\nabla\tilde\zeta^{(n-1)},\tilde v^{\theta(n-1)}, R^{(n)}b_0\cdot\nabla\tilde\Theta^{(n-1)}\right)}_{L^{\infty}_TH_{r,z}^3}^2\right.
\\&\quad\qquad\qquad\qquad\left.+\norm{\tilde\a^{(n)}}_{L^{\infty}_TH_{r,z}^2}^2\right).\nonumber
\end{align}
{\it Step 6: estimates of $(\tilde \nu^{(n)}, b_0\cdot\nabla\tilde\zeta^{(n)})$.} Finally, summing up the estimates \eqref{dest1}, \eqref{dest2}, \eqref{dpr}, \eqref{dest3}, and \eqref{dest4}, by using the similar argument used in Proposition \ref{fefe}, we obtain
\begin{equation}
\label{dp23}
\begin{split}
&\norm{\left(\tilde v^{r(n)}, \tilde v^{z(n)}, \tilde R^{(n)}, \tilde Z^{(n)}, b_0\cdot\nabla\tilde R^{(n)}, b_0\cdot\nabla\tilde Z^{(n)}\right)(t)}_3^2+\norm{\tilde\a^{(n)}(t)}_2^2\\&\leq P(M_0) T^2\left( \norm{\left(\tilde \nu^{(n)}, b_0\cdot\nabla\tilde\zeta^{(n)},\tilde\nu^{(n-1)},\tilde\zeta^{(n-1)},b_0\cdot\nabla\tilde\zeta^{(n-1)},\tilde v^{\theta(n-1)}, R^{(n)}b_0\cdot\nabla\tilde\Theta^{(n-1)}\right)}_{L^{\infty}_TH_{r,z}^3}^2\right.
\\&\quad\qquad\qquad\qquad\left.+\norm{\tilde\a^{(n)}}_{L^{\infty}_TH_{r,z}^2}^2\right).
\end{split}
\end{equation}
{\it Step 7: estimates for $\tilde v^{\theta(n)}, b_0\cdot\nabla\tilde\Theta^{(n)}$.} By doing a standard $D^3$ energy estimates for the system \eqref{diffe2}(see Section \ref{epes}), we can obatin
\begin{equation}
\label{thetaest}
\begin{split}
&\norm{\tilde v^{\theta(n)}, R^{(n+1)}b_0\cdot\nabla\tilde \Theta^{(n)}}_3^2+\norm{\tilde\Theta^{(n)}}_2^2\\&\leq P(M_0) T^2\left( \norm{\left(\tilde \nu^{(n)}, b_0\cdot\nabla\tilde\zeta^{(n)},\tilde\nu^{(n-1)},\tilde\zeta^{(n-1)},b_0\cdot\nabla\tilde\zeta^{(n-1)},\tilde v^{\theta(n-1)}, R^{(n)}b_0\cdot\nabla\tilde\Theta^{(n-1)}\right)}_{L^{\infty}_TH_{r,z}^3}^2\right.
\\&\quad\qquad\qquad\qquad\left.+\norm{\tilde\a^{(n)}}_{L^{\infty}_TH_{r,z}^2}^2\right).
\end{split}
\end{equation}
{\it Step 8: synthesis.}
We denote
\begin{equation}\label{normnorm}
\Psi^{(n)}:=\norm{\left(\tilde \nu^{(n)},\tilde\zeta^{(n)}, b_0\cdot\nabla\tilde\zeta^{(n)}, \tilde v^{\theta(n)}, R^{(n+1)}b_0\cdot\nabla\tilde\Theta^{(n)}\right)}_{L^{\infty}_TH_{r,z}^3}^2+\norm{\tilde\a^{(n)}, \tilde\Theta^{(n)}}_{L^{\infty}_TH_{r,z}^2}^2.
\end{equation}
Then by \eqref{dp23}, \eqref{thetaest} and taking $T$ sufficiently small (only depends on $M_0$), we have
\begin{equation}
\Psi^{(n)}\le \frac{1}{8}\left(\Psi^{(n-1)}+\Psi^{(n-2)}\right).
\end{equation}
This implies $\Psi^{(n)}\le P(M_0) 2^{-n}$, which yields the convergence of the sequence \newline $\{(v^r, v^z, v^{\theta}, R, Z, \Theta, q)^{(n)}\}_{n=0}$ to a limit $\{(v^r, v^z, v^{\theta}, R, Z, \Theta, q)\}$ in the norm of \eqref{normnorm} as $n\rightarrow\infty$.
\subsection{Local well-posedness of \eqref{eq:mhd}}
\begin{proof}[Proof of Theorem \ref{mainthm}]
The strong convergence of $\{(v^{(n)}, q^{(n)}, \eta^{(n)})\}$ is more than sufficient to pass to the limit as $n \rightarrow 0$ in \eqref{its1} and \eqref{its2} to produce a solution to  \eqref{eq:mhd}. The estimate \eqref{enesti} follows from \eqref{claim1}, while the uniqueness follows by the exactly same arguments as that of showing the convergence.
\end{proof}

\section*{Acknowledgement:}

The author would like to thank Professor Zhen Lei for his helpful discussion. The author would also like to thank the hospitality of the Shanghai Center of Mathematical Sciences. The author was in supported by NSFC (grant No. 11601305).


\begin{thebibliography}{99}
\bibitem{AD}
T. Alazard, J. M. Delort.
Global solutions and asymptotic behavior for two dimensional gravity water waves.
\emph{Ann. Sci. \'Ec. Norm. Sup\'er.} \textbf{48}  (2015), no. 5, 1149--1238.

 \bibitem{Alinhac}
S. Alinhac.
Existence d'ondes de rar\'efaction pour des syst\`emes quasi-lin\'eaires hyperboliques multidimensionnels.(French. English summary) [Existence of
rarefaction waves for multidimensional hyperbolic quasilinear systems]
\emph{Comm. Partial Differential Equations} \textbf{14} (1989), no. 2, 173--230.


\bibitem{Chen_08}
G. Chen, Y. Wang.
Existence and stability of compressible current-vortex sheets in three-dimensional magnetohydrodynamics.
\emph{Arch. Ration. Mech. Anal.} \textbf{187} (2008), no. 3, 369--408.
\bibitem{CL_00}
D. Christodoulou, H. Lindblad.
On the motion of the free surface of a liquid.
\emph{Comm. Pure Appl. Math.} \textbf{53} (2000), no. 12, 1536--1602.


\bibitem{CMST}
J. F. Coulombel, A. Morando, P. Secchi, P. Trebeschi. A priori estimates for 3D incompressible current-vortex sheets. \emph{Comm. Math. Phys.} \textbf{311} (2012), no. 1, 247--275.

\bibitem{CS07}
D. Coutand, S. Shkoller.
Well-posedness of the free-surface incompressible Euler equations with or without surface tension.
\emph{J. Amer. Math. Soc.}  \textbf{20} (2007), no. 3, 829--930.

\bibitem{DS_10}
D. Coutand, S. Shkoller.
A simple proof of well-posedness for the free-surface incompressible Euler equations.
\emph{Discrete Contin. Dyn. Syst. Ser. S.} \textbf{3} (2010), no. 3, 429--449.
\bibitem{Gu_2011}
X. Gu, Z. Lei.
Well-posedness of 1-D compressible Euler-Poisson equations with physical vacuum.
\emph{J. Differential Equations.} \textbf{252} (2012), 2160--2188.
\bibitem{GMS1}
P. Germain, N. Masmoudi, J. Shatah.
Global solutions for the gravity water waves equation in dimension 3.
\emph{Ann. of Math. (2)} \textbf{175} (2012), no. 2, 691--754.
\bibitem{GMS2}
P. Germain, N. Masmoudi, J. Shatah.
Global solutions for capillary waves equation.
\emph{Comm. Pure Appl. Math.} \textbf{68} (2015), no. 4, 625--687.
\bibitem{GW_16}
	X. Gu, Y. Wang.
	On the construction of solutions to the free-surface incompressible ideal magnetohydrodynamic equations.
	Preprint (2016), arXiv: 1609.07013
\bibitem{Go_04}
J. Goedbloed, S. Poedts.
Principles of magnetohydrodynamics with applications to laboratory
and astrophysical plasmas, Cambridge University Press, Cambridge, (2004).

\bibitem{Hao_13}
C. Hao, T. Luo.
A priori estimates for free boundary problem of incompressible inviscid magnetohydrodynamic flows.
\emph{Arch. Ration. Mech. Anal.} \textbf{212} (2014), no.3, 805--847.

\bibitem{IP}
A. Ionescu, F. Pusateri.
Global solutions for the gravity water waves system in 2D.
\emph{Invent. Math.} \textbf{199} (2015), no. 3, 653--804.

\bibitem{IP2}
A. Ionescu, F. Pusateri.
Global regularity for 2D water waves with surface tension.
\emph{Mem. Amer. Math. Soc.}, to appear.
\bibitem{Lannes}
D. Lannes.
Well-posedness of the water-waves equations.
\emph{J. Amer. Math. Soc.} \textbf{18} (2005), no. 3, 605--654.

\bibitem{Lindblad05}
H. Lindblad.
Well-posedness for the motion of an incompressible liquid with free surface boundary.
\emph{Ann. of Math. (2)} \textbf{162} (2005), no. 1, 109--194.


\bibitem{MasRou}
N. Masmoudi, F. Rousset.
Uniform regularity and vanishing viscosity limit for the free surface Navier-Stokes equations.
\emph{Arch. Ration. Mech. Anal.} (2016), DOI: 10.1007/s00205-016-1036-5.


\bibitem{Mo_14}
A. Morando, Y. Trakhinin, P. Trebeschi.
Well-posedness of the linearized plasma-vacuum interface problem in ideal incompressible MHD.
\emph{Quart. Appl. Math.} \textbf{72} (2014), no. 3, 549--587.


 \bibitem{N}
V. I. Nalimov.
The Cauchy-Poisson problem. (Russian)
\emph{Dinamika Splo$\breve{s}$n. Sredy Vyp. 18 Dinamika $\breve{Z}$idkost. so Svobod. Granicami.} \textbf{254} (1974), 104--210.




\bibitem{Secchi_13}
P. Secchi, Y. Trakhinin.
Well-posedness of the linearized plasma-vacuum interface problem.
\emph{Interfaces Free Bound.} \textbf{15} (2013), no. 3, 323--357.


\bibitem{Secchi_14}
P. Secchi, Y. Trakhinin.
Well-posedness of the plasma-vacuum interface problem.
\emph{Nonlinearity} \textbf{27} (2014), no. 3, 105--169.


\bibitem{SZ}
J. Shatah, C. Zeng.
Geometry and a priori estimates for free boundary problems of the Euler equation.
\emph{Comm. Pure Appl. Math.} \textbf{61} (2008), no. 5, 698--744.

\bibitem{Sun_15}
Y. Sun, W. Wang, Z. Zhang.
Nonlinear stability of current-vortex sheet to the incompressible MHD equations.
to appear in Comm. Pure Appl. Math.

\bibitem{Sun_17}
Y. Sun, W. Wang, Z. Zhang.
Well-posedness of the plasma-vacuum interface problem for ideal incompressible MHD.
Preprin (2017), arXiv:1705.00418.

\bibitem{Taylor}
\newblock M. Taylor,
\newblock Partial Differential Equations, Vol. I-III,
\newblock Berlin-Heidelberg-New York: Springer, (1996).


\bibitem{Trak_09}
Y. Trakhinin.
The existence of current-vortex sheets in ideal compressible magnetohydrodynamics.
\emph{Arch. Ration. Mech. Anal.} \textbf{191} (2009), no. 2, 245--310.


\bibitem{Trak_10}
Y. Trakhinin.
On the well-posedness of a linearized plasma-vacuum interface problem in ideal compressible MHD.
\emph{J. Differential Equations} \textbf{249} (2010), no. 10, 2577--2599


\bibitem{Wang_15}
Y. J. Wang, Z. Xin.
Vanishing viscosity and surface tension limits of incompressible viscous surface waves.
Preprint (2015), arXiv: 1504.00152.

\bibitem{Wu1}
S. Wu.
Well-posedness in Sobolev spaces of the full water wave problem in 2-D.
\emph{Invent. Math.} \textbf{130} (1997), no. 1, 39--72.


\bibitem{Wu2}
S. Wu.
Well-posedness in Sobolev spaces of the full water wave problem in 3-D.
\emph{J. Amer. Math. Soc.} \textbf{12} (1999), no. 2, 445--495.


\bibitem{Wu3}
S. Wu.
Almost global wellposedness of the 2-D full water wave problem.
\emph{Invent. Math.} \textbf{177} (2009), no. 1, 45--135.


\bibitem{Wu4}
S. Wu.
Global wellposedness of the 3-D full water wave problem.
\emph{Invent. Math.} \textbf{184} (2011), no. 1, 125--220.



\bibitem{ZZ}
P. Zhang, Z. Zhang.
On the free boundary problem of three-dimensional incompressible Euler equations.
\emph{Comm. Pure Appl. Math.} \textbf{61} (2008), no. 7, 877--940.

\end{thebibliography}
\end{document}